\documentclass[11pt,reqno]{amsart}
\usepackage{fullpage}
\usepackage{mathrsfs,amssymb,graphicx,verbatim,amsmath,amsfonts}
\usepackage{paralist}
\usepackage[breaklinks,pdfstartview=FitH]{hyperref}
\usepackage{upgreek}
\usepackage{tikz}
\usepackage[mathscr]{euscript}
\usepackage{relsize}
 \usepackage{amsaddr}
 
 \usepackage{graphicx}

\makeatletter
\@namedef{subjclassname@2020}{%
  \textup{2020} Mathematics Subject Classification}
\makeatother

\usepackage{dutchcal}

\DeclareSymbolFont{rmlargesymbols}{OMX}{mdbch}{m}{n}
\DeclareMathSymbol{\rmintop}{\mathop}{rmlargesymbols}{82}
\newcommand{\rmint}{\rmintop\nolimits}

\hypersetup{ocgcolorlinks=true,allcolors=testc}
\hypersetup{
     colorlinks   = true,
     citecolor    = black
}
\hypersetup{linkcolor=black}
\hypersetup{urlcolor=black}

\renewcommand{\ge}{\geqslant}
\renewcommand{\leq}{\leqslant}
\renewcommand{\geq}{\geqslant}
\renewcommand{\setminus}{\smallsetminus}
\renewcommand{\gamma}{\upgamma}
\allowdisplaybreaks

\usepackage{esint}
\usepackage[autostyle]{csquotes}

\renewcommand{\pi}{\uppi}

\newcommand{\e}{\varepsilon}
\newcommand{\R}{\mathbb R}

\newtheorem{theorem}{Theorem}
\newtheorem{lemma}[theorem]{Lemma}
\newtheorem{proposition}[theorem]{Proposition}

\newtheorem{corollary}[theorem]{Corollary}

\newtheorem{definition}[theorem]{Definition}

\theoremstyle{remark}
\newtheorem{remark}[theorem]{Remark}

\newtheorem{question}[theorem]{Question}

\renewcommand{\tau}{\uptau}

\renewcommand{\xi}{\upxi}
\renewcommand{\rho}{\uprho}

\renewcommand{\hat}{\widehat}

\newcommand{\N}{\mathbb N}

\newcommand{\eqdef}{\stackrel{\mathrm{def}}{=}}

\renewcommand{\theta}{\uptheta}
\renewcommand{\lambda}{\uplambda}

\renewcommand{\emptyset}{\varnothing}

\renewcommand{\gamma}{\upgamma}
\renewcommand{\beta}{\upbeta}
\renewcommand{\alpha}{\upalpha}
\renewcommand{\kappa}{\upkappa}
\renewcommand{\psi}{\uppsi}
\renewcommand{\rho}{\varrho}
\renewcommand{\delta}{\updelta}
\renewcommand{\pi}{\uppi}
\renewcommand{\omega}{\upomega}

\renewcommand{\eta}{\upeta}

\renewcommand{\kappa}{\upkappa}
\renewcommand{\mu}{\upmu}
\renewcommand{\nu}{\upnu}
\renewcommand{\pi}{\uppi}
\renewcommand{\zeta}{\upzeta}

\newcommand{\mb}{\mathbb}
\newcommand*\diff{\mathop{}\!\mathrm{d}}

\newcommand{\ms}{\mathscr}
\newcommand{\msf}{\mathsf}
\newcommand{\mr}{\mathrm}

\newcommand*{\Nearrow}{\rotatebox[origin=c]{45}{\(\Longrightarrow\)}}
\newcommand*{\Searrow}{\rotatebox[origin=c]{315}{\(\Longrightarrow\)}}

\begin{document}

\title{Concavity principles for weighted marginals}

\author{Dario Cordero-Erausquin and Alexandros Eskenazis}
\address{Institut de Math\'ematiques de Jussieu, Sorbonne Universit\'e, Paris, 75252, France}
\thanks{{\it E-mail addresses:} \href{mailto:dario.cordero@imj-prg.fr}{\nolinkurl{dario.cordero@imj-prg.fr}}, $\{$\href{alexandros.eskenazis@imj-prg.fr}{\nolinkurl{alexandros.eskenazis@imj-prg.fr}}, \href{ae466@cam.ac.uk}{\nolinkurl{ae466@cam.ac.uk}}$\}$.
}

\subjclass[2020]{Primary: 52A40; Secondary: 52A20, 28C20, 60D05, 47F10.}
\keywords{Brunn--Minkowski inequality, B-inequality, Gardner--Zvavitch problem, Borell--Brascamp--Lieb inequalities, symmetric convex sets, log-concave measures.}

\thanks{Part of this work was completed while A.~E.~was visiting the Hausdorff Institute for Mathematics during the Dual Trimester Program ``Boolean Analysis in Computer Science''.}

\maketitle

\vspace{-0.25in}

\begin{abstract}
We develop a general framework to study concavity properties of weighted marginals of $\beta$-concave functions on $\mathbb{R}^n$ via local methods. As a concrete implementation of our approach, we obtain a functional version of the dimensional Brunn--Minkowski inequality for rotationally invariant log-concave measures.  Moreover,  we derive a Pr\'ekopa-type concavity principle with rotationally invariant weights for even log-concave functions which encompasses the B-inequality.
\end{abstract}


\section{Introduction}


\subsection{Brunn--Minkowski inequalities and concavity principles}

The Brunn--Minkowski--Luster\-nik inequality asserts that for any nonempty Borel sets $A,B$ in $\R^n$,  
\begin{equation} \label{eq:bm}
\forall \ \lambda\in(0,1),\qquad \big| \lambda A+(1-\lambda)B\big|^{1/n} \geq \lambda |A|^{1/n}+(1-\lambda)|B|^{1/n},
\end{equation}
where $\lambda A+(1-\lambda)B = \{\lambda a+(1-\lambda) b: \ a\in A,\ b\in B\}$ is the Minkowski convex combination of $A$ and $B$ and $|\cdot|$ is the standard Lebesgue measure. Restricting the inequality to convex sets, which is the original framework studied by Brunn and Minkowski,  even allows for a concavity of the volume in the parameter $\lambda$. This statement, proven by Brunn for $n=2$,  is usually referred to as Brunn's \emph{concavity principle} and can be reformulated as follows: if $\Omega\subseteq\R^{n+1}$ is a bounded convex set and $\Omega_t = \{x\in\R^n: \ (t,x)\in \Omega\}$ for $t\in\R$, then the function $\msf{v}:\R\to\R_+$ given by
\begin{equation}
\forall \ t\in\R, \qquad \msf{v}(t) \eqdef \big| \Omega_t\big|^{1/n}
\end{equation}
is concave on its support. Such concavity statements suggest the existence of a \emph{local} or \emph{variational} take on the Brunn--Minkowski inequality, by computing the second derivative of $\msf{v}$, that is by understanding the second variation of volume around a convex set. This was the revolutionary approach taken by Hilbert \cite{Hil12},  upon which this work builds,  but with a functional perspective. 

An idea that emerged in the second half of the 20th century and whose importance cannot be overstated was to extend the geometric Brunn--Minkowski theory of sets to functions. 
In particular, the works of Borell \cite{Bor75} and Brascamp--Lieb \cite{BL76} (see also \cite{Rin76}) put forward the following powerful versions of the Brunn--Minkowski--Lusternik inequality: 
if $\kappa \in[-\frac{1}{n},\infty]$ and $f,g,h:\R^n\to\R_+$ are measurable functions such that
\begin{equation} \label{eq:bbl1}
\forall \ x,y\in\R^n, \qquad f(x)g(y)>0 \quad \Longrightarrow \quad h\big(\lambda x+(1-\lambda)y\big) \geq \big( \lambda f(x)^\kappa + (1-\lambda) g(y)^\kappa\big)^{1/\kappa}
\end{equation} 
for some $\lambda\in(0,1)$, then we also have
\begin{equation} \label{eq:bbl2}
\rmint_{\R^n} h(x)\,\diff x \geq \left( \lambda\left( \rmint_{\R^n} f(x)\,\diff x\right)^{\kappa_n} + (1-\lambda)\left( \rmint_{\R^n} g(x) \,\diff x\right)^{\kappa_n} \right)^{1/\kappa_n},
\end{equation}
where $\kappa_n = \frac{\kappa}{1+n\kappa} \in [-\infty,\frac{1}{n}]$. Per standard convention, the means on the right-hand sides of \eqref{eq:bbl1} and \eqref{eq:bbl2} are understood as the maximum when $\kappa=\infty$, as the $\lambda$-geometric mean when $\kappa=\kappa_n=0$, and as the minimum when $\kappa_n=-\infty$.  In the regime $\kappa>0$, which is most relevant to us,  an equivalent statement was proven earlier by Henstock--Macbeath \cite{HM53} for $n=1$ and Dinghas \cite{Din57} in general.  The case $\kappa=0$ coincides with the celebrated Pr\'ekopa--Leindler inequality \cite{Pre71,Pre73,Lei72b}. We refer to \cite{Gar02} for more on the history of these inequalities. 

The theorem of Borell and Brascamp--Lieb 
readily implies a functional concavity principle \emph{for marginals} when some {convexity} is imposed on the functions.  Let $\Omega \subseteq \R^{n+1}$ be a convex set and for each $t\in\R$, let $\Omega_t = \{x\in\R^n: \ (t,x)\in\Omega\}$ be again the corresponding section of $\Omega$.  
Then, if $\Phi:\Omega\to\R_+$ is a concave function and $\beta>0$, the function $\varphi:\R\to\R_+$ given by
\begin{equation} \label{eq:varphi}
\forall \ t\in \R,\qquad \varphi(t) \eqdef \left( \rmint_{\Omega_t} \Phi(t,x)^\beta\,\diff x\right)^{\frac{1}{\beta+n}}
\end{equation}
is concave on its support, assuming that the integral converges. Indeed, if $t_1,t_2\in \R$ are such that $\Omega(t_1)$ and $\Omega(t_2)$ are nonempty, the concavity of $\varphi$ follows by applying inequality \eqref{eq:bbl2} to $f=\Phi(t_1,\cdot)^\beta$, $g=\Phi(t_2,\cdot)^\beta$ and $h=\Phi(\lambda t_1 + (1-\lambda) t_2,\cdot)^\beta$, which satisfy \eqref{eq:bbl1} with $\kappa=\frac{1}{\beta}$.
Note that taking $\Phi(t,x)\equiv1$ on $\Omega$ and letting $\beta \to 0$,  one immediately recovers Brunn's concavity principle.  As a side remark,  we mention that given a convex function $V:\R^{n+1}\to\R\cup\{\infty\}$, one can apply the concavity principle for \eqref{eq:varphi} to the concave function $\Phi(t,x) = \big(1-\tfrac{1}{\beta}V(t,x)\big)_+$ and let $\beta\to\infty$ to recover the celebrated theorem of Pr\'ekopa \cite{Pre73}, \mbox{asserting that marginals of log-concave functions $e^{-V(t,x)}$ are log-concave.}


\subsection{Weighted Brunn--Minkowski inequalities and their functional versions}
A natural question is to what extent the previous inequalities for the Lebesgue measure, which we denoted by $|\cdot|$ or $\diff x$, are valid for a more general \emph{weight} $\mu$. In this respect, the log-concave case $\kappa=0$ of \eqref{eq:bbl1} and \eqref{eq:bbl2} is rather particular,  as products of log-concave functions are log-concave.  In other words, in Pr\'ekopa's theorem we can readily replace $\diff x$-marginals by weighted marginals with respect to any measure $\diff\mu(x)= e^{-V(x)}\, \diff x$, where $V:\R^n\to\R\cup\{\infty\}$ is convex. However, the situation is totally different in the range $\kappa>0$, as is confirmed by Borell's classification of convex measures \cite{Bor75}. 

A closely related major direction in contemporary Brunn--Minkowski theory aims to understand the role of central symmetry in volumetric inequalities for convex sets.  At the top of this hierarchy of questions lies the celebrated log-Brunn--Minkowski conjecture of B\"or\"oczky, Lutwak, Yang and Zhang \cite{BLYZ12} (see \cite{KM22,Mil21,IM24} for some recent developments), which, if true,  would have numerous consequences including the B-conjecture and the dimensional Brunn--Minkowski conjecture (see \cite{Sar16,LMNZ17}). These well-known open problems postulate that for every even log-concave measure $\mu$ on $\R^n$ and every symmetric convex sets $K,L$ in $\R^n$,  we have
\begin{equation} \label{eq:B-conj}
\forall \ a,b\in\R_+ \mbox{ and } \lambda\in(0,1), \qquad \mu\big(a^\lambda b^{1-\lambda}K\big) \geq \mu(aK)^\lambda \mu(bK)^{1-\lambda}
\end{equation}
and 
\begin{equation} \label{eq:dim-BM-conj}
\forall \ \lambda\in(0,1),\qquad \mu\big(\lambda K+(1-\lambda)L\big)^{1/n} \geq \lambda\mu(K)^{1/n}+(1-\lambda)\mu(L)^{1/n}
\end{equation}
respectively. Borell's classification of convex measures \cite{Bor75} implies that inequality \eqref{eq:dim-BM-conj} can be valid for \emph{all} convex sets in $\R^n$ only if $\mu$ is a multiple of the Lebesgue measure restricted on a convex set. In some sense, the dimensional Brunn--Minkowski conjecture predicts that Borell's classification breaks down when it comes to symmetric sets, allowing for the emergence of dimensional second-order phenomena going beyond log-concavity. 
A large body of work has been devoted to these conjectures in the last decades \cite{CFM04,GZ10,NT13,Mar16,Sar16,CLM17,LMNZ17,CR20,KL21,EM21,HKL21,KL22,KM22,BK22,Liv23,CR23}. The most general optimal result to date is due to \cite{CR23}, asserting in particular that \eqref{eq:B-conj} and \eqref{eq:dim-BM-conj}\mbox{ are both valid for rotationally invariant log-concave measures.}

The importance of the functional forms 
of the dimensional Brunn--Minkowski inequality \eqref{eq:bm} discussed earlier, naturally raises the question of whether the conjectures above admit functional versions as well.  The functional formulation of the B-conjecture was put forth in \cite{CR20}, where it was proven to be equivalent to the geometric version \eqref{eq:B-conj}.  The purpose of this paper is to investigate concavity principles in the spirit of the previous subsection for weighted marginals in the presence of symmetry.  Our first main result,  is the following functional formulation of the dimensional Brunn--Minkowski conjecture \eqref{eq:dim-BM-conj} for rotationally invariant measures.

\begin{theorem} \label{thm:bbl}
Let $w:[0,\infty)\to(-\infty,\infty]$ be an increasing function such that $t\mapsto w(e^t)$ is convex on $\R$ and let $\mu$ be the measure on $\R^n$ with density $e^{-w(|x|)}$.  Moreover, let $\Omega\subseteq \R^{n+1}$ be a convex set such that for each $t\in\R$, the set $\Omega_t=\{x\in\R^n: \ (t,x)\in\Omega\}$ is symmetric, and $\Phi:\Omega\to\R_+$ a concave function such that for each $t$, the function $\Phi(t,\cdot)$ is even on $\Omega_t$.  Then, 
\begin{equation} \label{eq:varphi-thm}
 \varphi(t) \eqdef \Bigg( \rmint_{\Omega_t} \Phi(t,x)^\beta\,\diff\mu(x)\Bigg)^{\frac{1}{\beta+n}}
\end{equation}
is concave on its support for every $\beta>0$, provided that the integral converges.
\end{theorem}

Note that, as in~\cite{CR23}, the class of permissible densities  of the measure $\mu$ include log-concave ones like $e^{-|x|^\alpha}$, $\alpha\geq1$, but also numerous which are not log-concave like $(1+|x|^\alpha)^{-\beta}$ with $\alpha,\beta>0$.  Up to scaling,  the statement of Theorem \ref{thm:bbl} is equivalent to the fact that for any convex function $V:\R^{n+1}\to\R\cup\{\infty\}$ such that each $V(t,\cdot)$ is even on $\R^n$, the function
\begin{equation}
\varphi(t) \eqdef \Bigg( \rmint_{\R^n} \big(1-V(t,x)\big)_+^\beta \, \diff\mu(x) \Bigg)^{\frac{1}{\beta+n}}
\end{equation}
is concave on its support for every $\beta>0$, provided that the integral converges.

Theorem~\ref{thm:bbl} should be viewed as the functional version of the dimensional Brunn--Minkowski inequality for $\mu$, proven in \cite{EM21,CR23}, which corresponds to $\beta\to0$.  The possibility of such functional inequalities was first investigated by Roysdon and Xing in \cite{RX21}. There, they proposed a general conjecture \cite[Conjecture~6.2]{RX21} aiming to generalize the Borell--Brascamp--Lieb inequality to arbitrary measures. This can easily be seen to be false in general (take,  in their notation, $p=1$, $t=1/2$, $f={\bf 1}_{\msf{B}_2^n+Me_1}$, $g={\bf 1}_{\msf{B}_2^n-Me_1}$ and $h={\bf 1}_{B_2^n}$ for any finite measure $\mu$ and let $M\to\infty$). However, Theorem \ref{thm:bbl} shows that it is true for functions and measures with appropriate concavity properties.

Specializing to functions $\Phi(t,x)$ which do not depend on $t$, one obtains the following corollary.

\begin{corollary} \label{cor:b-conc}
Let $w:[0,\infty)\to(-\infty,\infty]$ be an increasing function such that $t\mapsto w(e^t)$ is convex on $\R$ and let $\mu$ be the measure on $\R^n$ with density $e^{-w(|x|)}$.  Moreover, let $C\subseteq\R^n$ be a symmetric convex set and consider an even concave function $\Phi:C\to \R_+$. Then, for every $\beta>0$, the measure $\diff\nu_\beta(x) = \Phi(x)^\beta\,\diff\mu(x)$, supported on $C$, satisfies the Brunn--Minkowski inequality
\begin{equation}
\forall \ \lambda\in(0,1),\qquad \nu_\beta\big(\lambda K + (1-\lambda)L\big)^{\frac{1}{\beta+n}} \geq \lambda\nu_\beta(K)^{\frac{1}{\beta+n}}+(1-\lambda)\nu_\beta(L)^{\frac{1}{\beta+n}}
\end{equation}
for every symmetric convex sets $K,L$ in $\R^n$.
\end{corollary}

Indeed, Corollary \ref{cor:b-conc} is an immediate application of Theorem \ref{thm:bbl} for the convex set
\begin{equation}
\Omega\eqdef \bigcup_{t\in[0,1]} \{t\}\times \big((tK+(1-t)L)\cap C\big) \subseteq \R^{n+1}
\end{equation}
and the concave function $\Phi(t,x)\eqdef \Phi(x){\bf 1}_C(x)$, where $(t,x)\in\Omega$.

It is worth emphasizing that when $\mu$ is a log-concave measure,  the measures $\nu_\beta$ appearing in Corollary \ref{cor:b-conc} are also log-concave and then this result fits the framework of the dimensional Brunn--Minkowski conjecture. In particular,  Corollary \ref{cor:b-conc} does not follow from the best known general bound in this conjecture due to Livshyts \cite{Liv23} in the regime $\beta\lesssim n^4$. 

The concavity principles for the \emph{unweighted} marginal function \eqref{eq:varphi} were first investigated via local methods in the work of Nguyen \cite{Ngu14a}, under the name \emph{dimensional Pr\'ekopa theorems}. In \cite{Ngu14a}, Nguyen claims to prove the concavity principles both in the regime $\beta>0$ of Theorem \ref{thm:bbl} and in the dual regime $\beta<0$ which corresponds to $\kappa\in[-\tfrac{1}{n},0)$ in the Borell--Brascamp--Lieb inequalities \eqref{eq:bbl1} and \eqref{eq:bbl2}. Central to Nguyen's argument is a convenient expression for the second derivative of the marginal function \eqref{eq:varphi} when the underlying set $\Omega = I\times U$ for an interval $I$ in $\R$ and a convex set $U$ in $\R^n$, much like in  Brascamp and Lieb's local proof \cite{BL76} of Pr\'ekopa's theorem. This expression, inspired by \cite{Cor05,CFM04}, is obtained by a delicate application of H\"ormander's $L_2$ method. 

While restricting to product sets $\Omega$ can easily be shown to be sufficient to yield the general concavity principles in the range $\beta\leq0$ by a simple approximation argument, the same reasoning (details for which are omitted in \cite[p.~26]{Ngu14a}) appears to fail fundamentally in the dual range $\beta>0$.  For instance, the distribution of an indicator function ${\bf 1}_{\Omega}(t,x)$, where $\Omega$ is  a general convex subset of $\R^{n+1}$ cannot be approximated by those of functions of the form $\Phi(t,x) {\bf 1}_{I\times U}(t,x)$, where $\Phi$ is a concave nonnegative function on $I\times U$, as the support of such functions will always be a product set. In order to overcome this issue and to prove Theorem \ref{thm:bbl}, we are forced to find a convenient expression for the derivatives of the function \eqref{eq:varphi-thm} without making additional assumptions for the underlying domain $\Omega$. This task inherently requires combining the $L_2$ reasoning developed in \cite{Cor05,CFM04,Ngu14a} for the purpose of studying concavity principles for marginals with geometric considerations developed by Kolesnikov and Milman in \cite{KM18} in the context of weighted Brunn--Minkowski inequalities in Riemannian manifolds. This mixed geometric-functional application of the $L_2$ method appears to be new and we expect it to be useful in other contexts as well. We refer to Section \ref{sec:2} for a comparison of the two implementations of the $L_2$ method, for sets versus for functions,  and some perspective on Bochner--H\"ormander's $L_2$ method in modern Brunn--Minkowski theory.


\subsection{Poincar\'e inequalities associated to weighted concavity principles}

In the pioneering article \cite{BL76}, Brascamp and Lieb showed that  Pr\'ekopa's theorem follows 
from a fundamental Poincar\'e-type inequality for log-concave measures, recalled below in~\eqref{eq:bl-var} and often referred to as the Brascamp--Lieb variance inequality. In fact, the Brascamp--Lieb inequality is equivalent to Prékopa's theorem, see e.g. \cite{CK12}. A proof of the Brascamp--Lieb inequality using the Prékopa--Leindler inequality, inspired by an argument of Maurey \cite{Mau91}, was given by Bobkov and Ledoux \cite{BL00} and extended to dimensional functional inequalities in~\cite{BL09a}. In a similar spirit,  Nguyen \cite{Ngu14b} investigated Poincar\'e-type inequalities which correspond to the concavity principles for \eqref{eq:varphi}.
Building upon this line of research, we deduce the following Poincar\'e--Brascamp--Lieb inequalities for weighted $\beta$-concave measures and even functions via Theorem \ref{thm:bbl}.

\begin{theorem} \label{thm:poincare}
Let $w:[0,\infty)\to(-\infty,\infty]$ be an increasing function such that $t\mapsto w(e^t)$ is convex on $\R$ and let $\mu$ be the measure on $\R^n$ with density $e^{-w(|x|)}$.  
Let also $C\subseteq\R^n$ be a symmetric convex set and consider an even concave $C^2$ function $\Phi:C\to \R_+$ such that $\diff\nu_\beta(x) = \Phi(x)^\beta\,\diff\mu(x)$ is a probability measure supported on $C$.  Then, for any even locally Lipschitz function $f\in L_2(\nu_\beta)$, setting $g=f\Phi$, we have the inequality
\begin{equation} \label{eq:new-poin}
(\beta-1) \mathrm{Var}_{\nu_\beta}(f) \leq \rmint_C \frac{\big\langle (-\nabla^2 \Phi) \nabla g, \nabla g\big\rangle}{\Phi}\,\diff\nu_\beta + \frac{n}{\beta+n} \left(\rmint_{C} f\,\diff\nu_\beta\right)^2.
\end{equation}
\end{theorem}

When $\mu$ is the Lebesgue measure, \eqref{eq:new-poin} coincides with the endpoint $r=0$ of Nguyen's \cite[Theorem~2]{Ngu14b} under 
symmetry assumptions for $f$.  It is clear that variants of our methods can give weighted versions of \cite[Theorem~2]{Ngu14b} for arbitrary values of $r$ but, in lack of a specific application and for the sake of conciseness, \mbox{we defer these investigations to future works.}

Let us mention that instead of deriving Theorem~\ref{thm:poincare} from Theorem \ref{thm:bbl}, one can alternatively  give a direct proof of it using the $L_2$ arguments of  the proof of Theorem \ref{thm:bbl}, as in~\cite{Ngu14b}. Since our main emphasis in the present paper is on concavity principles, we chose not to repeat this reasoning or to discuss how to extract the relevant arguments for a direct proof.


\subsection{Log-concavity principles for weighted marginals}

Theorem \ref{thm:bbl} is a satisfactory weighted extension of the dimensional concavity principle for~\eqref{eq:varphi}
with respect to rotationally invariant log-concave measures under evenness assumptions. However, 
letting $\beta\to\infty$ reduces it to preservation of log-concavity, as in Pr\'ekopa's theorem. Therefore, it fails to capture more refined log-concavity properties which may crucially rely on central symmetry. The prime example of such a property is the functional version of the B-conjecture \cite{CR20},  postulating that if $V,W:\R^n\to\R\cup\{\infty\}$ are even convex functions, then the function $\alpha:\R\to\R_+$ given by
\begin{equation} \label{eq:fun-B}
\forall \ t\in\R,\qquad \alpha(t) \eqdef \rmint_{\R^n} e^{-V(e^tx)-W(x)}\,\diff x
\end{equation} 
is log-concave.  In view of the result of \cite{CR23}, this property holds when $W$ is further assumed to be rotationally invariant.  
Our next theorem offers an extension of the B-inequality \eqref{eq:B-conj} in this case. 

\begin{theorem}
\label{cor:prekopa}
Fix $\kappa\in[0,1]$.   Let $w:[0,\infty)\to(-\infty,\infty]$ be an increasing function such that $t\mapsto w(e^t)$ is convex on $\R$ and let $\mu$ be the measure on $\R^n$ with density $e^{-w(|x|)}$.   Moreover, let $V:\R^{n+1}\to\R$ be a $C^2$ function such that for each $t$, the\mbox{ function $V(t,\cdot)$ is even on $\R^n$ and}
\begin{equation} \label{eq:assum-pr-cor}
\forall \ (t,x)\in\R^{n+1},\qquad \nabla^2V(t,x) \succeq \kappa \begin{pmatrix}
      \langle \nabla_x V(t,x), x\rangle &   \nabla_xV(t,x)^\ast \\
    \nabla_x V(t,x)  & {\bf 0}_{n\times n}
    \end{pmatrix}.
\end{equation}
Then, the function given by
\begin{equation}
\forall \ t\in\R,\qquad  \alpha(t) \eqdef  \rmint_{\R^n} e^{-V(t,x)}\,\diff\mu(x)
 \end{equation}
is log-concave.
\end{theorem} 

With $\kappa=0$, this result provides some new log-concavity preservation in the even case. When $\mu$ is itself log-concave (in the notation of the Theorem, this corresponds to the case where $w$ is convex, which is stronger than the convexity of $t\mapsto w(e^t)$), this follows without the evenness assumption from Pr\'ekopa's theorem \cite{Pre73}, as we already mentioned,  because products of log-concave functions are log-concave. However, such a self-improvement property is not a priori evident for non-log-concave $\mu$, and for this central symmetry plays a key role, presumably.

With $\kappa=1$, Theorem \ref{cor:prekopa}  in fact contains the functional B-theorem \eqref{eq:fun-B}, as the Hessian of the function given by $\R^{n+1}\ni (t,x)  \mapsto \Psi(t,x)\eqdef V(e^tx)$ satisfies
 \begin{equation} \label{eq:b-hessian}
 \begin{split}
 \big\langle  \nabla^2 \Psi(t,x)\cdot &(u_0,\bar{u}), (u_0,\bar{u})\big\rangle 
 \\ & = e^{2t}\big\langle \nabla^2 V(e^tx)\cdot (u_0x+\bar{u}), (u_0x+\bar{u})\big\rangle + 2u_0 e^t \langle \nabla V(e^tx),\bar{u}\rangle + u_0^2 e^t \langle V(e^tx), x\rangle
 \\ & = e^{2t}\big\langle \nabla^2 V(e^tx)\cdot (u_0x+\bar{u}), (u_0x+\bar{u})\big\rangle + 2u_0 \langle \nabla_x\Psi(t,x),\bar{u}\rangle + u_0^2 \langle \nabla_x\Psi(t,x),x\rangle,
 \end{split}
 \end{equation}
for every $(u_0,\bar{u})\in\R^{n+1}$. 

The existence of a unified approach to Theorem \ref{thm:bbl} and Theorem \ref{cor:prekopa} offers a conceptual explanation for the fact that the refined spectral inequalities for rotationally invariant measures discovered in \cite{CR23} readily yielded both the B-theorem and the dimensional Brunn--Minkowski inequality for such measures.  Moreover, it highlights a functional inequality inspired by \cite{EM21} which, if true, would imply both conjectures for general even measures (see Question \ref{q:ce}).  Last but not least, as shall be explained in Remark \ref{rem:B-proofs}, the proof of Theorem \ref{cor:prekopa} in the case $\kappa=1$ which implies the functional B-theorem \eqref{eq:fun-B}, also yields a new proof of the B-theorem for such measures (see also Remark \ref{rem:B-BL} and the Appendix for a \emph{very old} proof of the Gaussian B-theorem).


\subsection*{Structure of the paper} In Section \ref{sec:2}, we present some preliminaries on H\"ormander's $L_2$ method, along with a concise discussion of its applications in Brunn--Minkowski theory. In Section \ref{sec:comp}, we use the $L_2$ method to derive a general formula for the second derivative of weighted marginals in terms of solutions to certain elliptic PDE, while in Sections \ref{sec:main} and \ref{sec:log-conc} we combine this formula with spectral estimates to deduce our weighted concavity principles, Theorems \ref{thm:bbl} and \ref{cor:prekopa}. Finally, in Section \ref{sec:poincare}, we use Theorem \ref{thm:bbl} to derive the weighted Poincar\'e-type inequality of Theorem \ref{thm:poincare} and in Section \ref{sec:last} we present some further remarks and open questions that arise from this work.


\subsection*{Related work} While the present paper was in final stages of preparation, the preprint \cite{AL25} of Aishwarya and Li appeared on arXiv, containing results closely related to Theorem \ref{thm:bbl} in the special case that the rotationally invariant measure $\mu$ is the standard Gaussian.  More specifically, Theorem 1 is \emph{equivalent} to a $\mu$-weighted analogue of the Borell--Brascamp--Lieb inequalities \eqref{eq:bbl1}, \eqref{eq:bbl2} when the even functions $f^\kappa$ and $g^\kappa$ are assumed to be concave on their supports.  In contrast to the local $L_2$ method that we develop here to tackle these problems, the authors of \cite{AL25} use a refreshing information theoretic approach to derive the same conclusion when $f$ and $g$ are assumed to be merely log-concave,  yet only in the special case that $\mu$ is the standard Gaussian measure.


\subsection*{Acknowledgements} We are very grateful to Emanuel Milman for constructive feedback and to Tomasz Tkocz for many useful discussions. 


\section{H\"ormander's $L_2$ method in Brunn--Minkowski theory} \label{sec:2} 


\subsection{Generalities on the $L_2$ method} Let $\nu$ be a probability measure on $\R^n$ whose density $e^{-V}:\R^n\to\R_+$ is $C^2$ smooth in its support. Throughout this section, we shall tacitly assume that $U=\mr{supp}(\nu)=\{x: \ V(x)<\infty\}$ is an open convex set in $\R^n$ with $C^2$ smooth boundary $\Sigma$. Associated to $\nu$ is the elliptic operator $\ms{L}_\nu$, whose action on a $C^2$ function $u:U\to\R$ is given by
\begin{equation}
\ms{L}_\nu u = \Delta u - \langle \nabla V, \nabla u\rangle.
\end{equation}
For a pair of $C^2$ smooth functions $u,v$, this operator satisfies the integration by parts formula
\begin{equation} \label{eq:bp}
\rmint_U \ms{L}_\nu u \cdot v \,\diff\nu = -\rmint_U \langle \nabla u, \nabla v\rangle\,\diff\nu+\rmint_{\Sigma} \frac{\partial u}{\partial \hat{\nu}_{\Sigma}}\cdot v \,\diff\nu_{\Sigma}.
\end{equation}
Here we denoted by $\frac{\partial}{\partial \hat{\nu}_{\Sigma}}$ the outward pointing normal derivative on $\Sigma$ and by $\nu_{\Sigma}\equiv\nu$ the measure with density $e^{-V}$ with respect to the $(n-1)$-dimensional Hausdorff measure on $\Sigma$. In particular, this implies that $\ms{L}_\nu$ can be extended to a self-adjoint operator on $L_2(\nu)$ with dense domain $\mathcal{D}$ corresponding to Neumann condition on the boundary so that \eqref{eq:bp} still holds.

As the potential $V$ and the underlying domain $U$ are $C^2$ smooth, all the classical elliptic existence, uniqueness and regularity results \cite{GT01} apply to our setting. In particular, suppose that $f:U\to\R$ and \mbox{$\psi:\Sigma\to\R$ are two $C^1$ smooth functions. Then, the equation}
\begin{equation}
\ms{L}_\nu u =f \ \  \mbox{ in } U
\end{equation} \label{eq:compat}
has a $C^2$ solution $u$ with Neumann boundary data $\frac{\partial u}{\partial \hat{\nu}_\Sigma} = \psi$ under the compatibility condition
\begin{equation}
\rmint_U f \,\diff\nu = \rmint_{\Sigma} \psi\,\diff \nu.
\end{equation}
In particular, the equation $\ms{L}_\nu u = f$ admits a solution satisfying $\frac{\partial u}{\partial \hat{\nu}_\Sigma} \equiv 0$ provided that $\rmint_U f\,\diff\nu=0$. We refer to \cite[Theorem~2.5]{KM18} for a more precise statement involving H\"older regularity classes.

H\"ormander developed a powerful method for solving the $\overline{\partial}$ equation (see \cite{Hor65,Hor94}), while emphasizing the role of convexity (or plurisubharmonicity) of the underlying domain and the potential $V$. Since H\"ormander's seminal work, his $L_2$ method has also become an indespensable tool to obtain spectral inequalities when combined with variants of the classical Bochner formula from differential geometry. Due to these applications, the name \emph{Bochner method} is also frequently encountered in the literature referring to the $L_2$ method. Put simply, the $L_2$ method for spectral inequalities amounts to the following: to prove a Poincar\'e-type inequality for a function $f$ with respect to a measure $\nu$, solve the equation $\ms{L}_\nu u =f$, prove the dual inequality for $u$ using a variant of the Bochner formula, and dualize back to deduce the desired inequality for $f$. We refer to \cite[Lemma~1]{BC13} for a precise formulation and to \cite{CFM04,Cor05,CK12,BC13,Ngu14b,KM17,KM18,KM22,CR23} for examples of applications.

For our purposes, the relevant Bochner-type identity is a weighted version of Reilly's formula for manifolds with boundary \cite{Rei77}, see \cite{MD10,KM17}. It asserts that if $u:\overline{U}\to\R$ is a $C^2$ function with normal derivative $\psi = \frac{\partial u}{\partial \hat{\nu}_\Sigma}$, then 
\begin{equation} \label{eq:bochner}
\begin{split}
\rmint_U \big( \ms{L}_\nu u\big)^2\,\diff\nu = \rmint_U \|\nabla^2u\|_{\mr{HS}}^2 & + \langle \nabla^2 V\nabla u,
\nabla u\rangle \,\diff \nu
\\& +\rmint_\Sigma \msf{H}_{\nu,\Sigma} \cdot \psi^2 + \mr{I\!I}(\nabla_\Sigma u,\nabla_\Sigma u) - 2\langle \nabla_\Sigma u, \nabla_\Sigma \psi\rangle\,\diff\nu,
\end{split}
\end{equation}
where $\nabla_\Sigma u$ is the tangential gradient of $u$ on $\Sigma$,  $\mr{I\!I}(\cdot,\cdot)$ is the second fundamental form of $\Sigma$ and $\msf{H}_{\nu,\Sigma} \eqdef \mathrm{tr}\mr{I\!I} - \langle \nabla V, \hat{\nu}_\Sigma\rangle$ is the $\nu$-weighted mean curvature of $\Sigma$.
We refer to \cite[Theorem~2.1]{KM18} for a more general version of the formula which is valid for functions on a Riemannian manifold with arbitrary Neumann boundary data. Dualizing \eqref{eq:bochner} with $\psi\equiv0$ and using \cite[Lemma~1]{BC13}, one immediately deduces the celebrated Brascamp--Lieb variance inequality:~if $\nu$ is strictly log-concave, i.e.~$\nabla^2 V\succ {\bf 0}_{n\times n}$ on $\R^n$, then every locally Lipschitz function $f\in L_2(\nu)$ satisfies
\begin{equation} \label{eq:bl-var}
\mathrm{Var}_\nu(f) \leq \rmint_{\R^n} \big\langle \big(\nabla^2 V\big)^{-1} \nabla f, \nabla f\big\rangle\,\diff\nu.
\end{equation}

In the context of Brunn--Minkowski-type inequalities and their functional versions,  one wants to study the concavity properties of functions of the form
 \begin{equation} \label{eq:gen-conca}
 \alpha(t) \eqdef \rmint_{\Omega_t} e^{-G(t,x)}\,\diff\mu(x),
 \end{equation}
for some smooth family $\{\Omega_t\}$ of convex sets in $\R^n$ and a smooth function $G$ defined on a subset of $\R^{n+1}$.  The $L_2$ method has been applied to this setting in two fairly distinct ways, depending on the nature of the sets $\Omega_t$ and the function $G$,  which we shall now proceed to describe.


\subsection{Local proofs of geometric inequalities} The Brunn--Minkowski inequality for convex sets corresponds to the $\frac1n$-concavity of the function \eqref{eq:gen-conca} on its support when $\mu$ is the Lebesgue measure, $G\equiv0$ and $\Omega_t = (1-t)K+tL$ for two convex sets $K,L$ in $\R^n$, where $t\in[0,1]$. The first infinitesimal proof of this inequality originates in celebrated work of Hilbert (see \cite{Hil12} or \cite{Hor94}).  This important approach was revisited by Colesanti \cite{Col08}, who, conversely, used the Brunn--Minkowski inequality to derive an equivalent Poincar\'e-type inequality for functions defined on the boundary of convex sets.  Morally,  the sets $\{\Omega_t\}_{t\in(0,\eta)}$ are realized as infinitesimal perturbations of $K$ by a vector field of the form $\psi\cdot\hat{\nu}_{\partial K}$ and thus the resulting concavity can be expressed as a functional inequality for $\psi:\partial K\to\R$.  Finally, a direct proof of Colesanti's inequality (in the more general setting of weighted Riemannian manifolds), and thus a new local proof of the Brunn--Minkowski inequality for convex sets, was discovered by Kolesnikov and Milman in \cite{KM18} using the Bochner method.

Let $K$ be an open convex set in $\R^n$ and $\psi:\partial K\to\R$ a $C^1$ smooth function. To prove Colesanti's inequality for $\psi$, Kolesnikov and Milman consider a $C^2$ smooth function $u:K\to \R$ such that 
\begin{equation} \label{eq:km}
\begin{cases}
\Delta u\equiv\kappa, & \mbox{in } K \\
\frac{\partial u}{\partial \hat{\nu}_{\partial K}} = \psi, & \mbox{on } \partial K
\end{cases},
\end{equation}
where $\kappa$ is an appropriate constant so that the compatibility condition \eqref{eq:compat} is satisfied. Afterwards, they use the Bochner formula \eqref{eq:bochner} to prove a dual to Colesanti's inequality for $u$ and this completes the proof. This approach was also extended to prove dimensional Brunn--Minkowski inequalities for measures $\mu$ in a series of more recent works \cite{KL21,EM21,CR23} by replacing $\Delta$ with the operator $\ms{L}_\mu$.


\subsection{Local proofs of concavity principles} When it comes to the \emph{functional} counterparts of geometric inequalities, the first application of infinitesimal methods can be traced back to Brascamp and Lieb \cite[Theorem~4.2]{BL76} who used inequality \eqref{eq:bl-var} to give an alternative proof of Pr\'ekopa's theorem \cite{Pre73}. The first applications of the $L_2$ method towards concavity principles are due to \cite{CFM04, Cor05} for the functional Gaussian B-inequality and (complex) Pr\'ekopa's theorem respectively, and were partly inspired by a local proof of Pr\'ekopa's theorem due to Ball, Barthe and Naor \cite{BBN03}.  This $L_2$ approach was later extended to encompass the dimensional versions of Pr\'ekopa's theorem due to Borell and Brascamp--Lieb \cite{Bor75,BL76} in the works of Nguyen \cite{Ngu14a,Ngu14b}.

Pr\'ekopa's theorem asserts that the marginal function \eqref{eq:gen-conca} is log-concave on $\R$ when $\mu$ is the Lebesgue measure,  $\Omega_t=\R^n$ and $G:\R^{n+1}\to\R$ is a convex function.  Following the exposition of the proof of \cite{Cor05} in \cite[Introduction]{CK12}, to prove that $(-\log\alpha)''(t)\geq0$ one considers the log-concave probability measure $\nu_t$ on $\R^n$ with density proportional to $e^{-G(t,\cdot)}$ and needs to prove a Poincar\'e-type inequality for the function $\partial_t G(t,\cdot)$ with respect to $\nu_t$. To this end, the works \cite{Cor05,CK12} then consider a $C^2$ smooth function $u_t:\R^{n}\to\R$ satisfying the elliptic PDE
\begin{equation} \label{eq:subforb}
\ms{L}_{\nu_t} u_t = \partial_tG(t,\cdot) - \rmint_{\R^n} \partial_t G(t,\cdot) \,\diff\nu_t
\end{equation}
and prove the dual inequality for $u_t$ using the Bochner formula \eqref{eq:bochner} on $\R^n$.  Here,  a zero Neumann boundary condition at infinity is again implicitly imposed.

In the subsequent works \cite{Ngu14a,Ngu14b} by Nguyen, which deal with the dimensional versions of Pr\'ekopa's theorem, one is forced to prove concavity principles for functions defined on convex subsets of $\R^n$.  To prove that if $\Phi(t,x)$ is a concave function on a product set $I\times U$, then the same holds for the marginal \eqref{eq:varphi}, Nguyen considers the probability measure $\nu_t$ with density proportional to $\Phi(t,\cdot)^\beta$ in $U$ and a $C^2$ function $u_t:\overline{U}\to \R$ satisfying the elliptic PDE with Neumann boundary condition
\begin{equation} \label{eq:ce}
\begin{cases}
\ms{L}_{\nu_t} u_t = \frac{\partial_t \Phi(t,\cdot)}{\Phi(t,\cdot)} - \rmint_U \frac{\partial_t \Phi(t,\cdot)}{\Phi(t,\cdot)}\,\diff\nu_t, & \mbox{in } U \\
\frac{\partial u}{\partial \hat{\nu}_{\partial U}} \equiv 0, & \mbox{on } \partial U
\end{cases}.
\end{equation}
Afterwards, the dual inequality for $u_t$ is proven using the weighted Bochner-type formula \eqref{eq:bochner}.


\subsection{Our contribution}
We draw the reader's attention to these two very distinct implementations of the $L_2$ method in Brunn--Minkowski theory:~\eqref{eq:km} has been used to give infinitesimal proofs of geometric inequalities and \eqref{eq:subforb}, \eqref{eq:ce} for their functional counterparts when the underlying domain is a product set.  As already explained in the introduction, Nguyen's reasoning from \cite{Ngu14a,Ngu14b} appears insufficient to recover the case $\beta>0$ of the concavity principle for the marginal \eqref{eq:varphi} in full generality as one cannot a priori assume that the underlying domain $\Omega$ has a product structure $I\times U$.  To this end,  the  present work contains a twofold contribution in our understanding of $L_2$ methods in convexity.  First, we introduce a mixture of the geometric and functional implementations of the $L_2$ method to express the second derivative of the function \eqref{eq:gen-conca} when both the sets $\{\Omega_t\}$ and the function $G$ are nontrivial.  Secondly, we apply these considerations to the study of \emph{weighted} Brunn--Minkowski inequalities under symmetry assumptions, showing that the spectral tools used in \cite{EM21,CR23} suffice to yield the local versions of the weighted concavity principles presented in Theorem \ref{thm:bbl}.  As a side-result,this new perspective also yields a new proof of the B-inequality for rotationally invariant measures via the more general concavity principle of Theorem \ref{cor:prekopa}.


\section{A formula for the second derivative of marginals} \label{sec:comp}

In this section, we present a formula for the second derivative of marginals using the $L_2$ method.  Our computation relies on a mixture of the methods of Nguyen \cite[Theorem~1.1]{Ngu14a} and Kolesnikov--Milman \cite[Section~3.1]{KM18}, albeit with several substantial modifications which are necessary to treat the combination of the functional and geometric settings treated here.


\subsection{Perturbations of convex sets} Before proceeding to the actual computation, we present some preliminaries related to the variation of measures of convex sets under small perturbations.  Let $\{\Omega_t\}_{t\in(-\eta,\eta)}$ be a family of bounded convex sets in $\R^n$ with nonempty interiors and assume that each $\Omega_t$ is strictly convex with $C^2$ boundary $\Sigma_t$. Moreover, let $g:(-\eta,\eta)\times\Sigma_0\to\R$ be such that
\begin{equation} \label{eq:defn-g}
\forall \ (t,y)\in(-\eta,\eta)\times\Sigma_0,\qquad h_{\Omega_t} \big(\hat{\nu}_{\Sigma_0}(y)\big) = g(t,y),
\end{equation}
where as usual we denote by $h_L:\mb{S}^{n-1}\to\R$ the support function of a convex body $L$ in $\R^n$ and by $\hat{\nu}_L:\partial L\to\mb{S}^{n-1}$ its outward pointing unit normal vector.  We shall assume that $g$ is jointly $C^2$ smooth. Moreover, consider the mapping $F:(-\eta,\eta)\times \Sigma_0  \to \R^n$ given by
\begin{equation} \label{eq:defn-F}
\forall \ (t,y)\in(-\eta,\eta)\times\Sigma_0,\qquad F(t,y) = \hat{\nu}_{\Sigma_t}^{-1} \circ \hat{\nu}_{\Sigma_0}(y).
\end{equation}
Observe that due to the strict convexity and regularity assumptions imposed on $\Omega_t$,  $F_0\eqdef F(0,\cdot) \equiv \mathrm{Id}$ and each $F_t \eqdef F(t,\cdot): \Sigma_0\to\Sigma_t$ is a diffeomorphism that satisfies the compatibility equation
\begin{equation} \label{eq:compat}
\forall \ (t,y)\in(-\eta,\eta)\times\Sigma_0,\qquad \hat{\nu}_{\Sigma_t} \big( F(t,y) \big) = \hat{\nu}_{\Sigma_0}(y).
\end{equation}  
Its $t$-derivative is calculated in the following lemma which builds upon \cite[Proposition~6.3]{KM18}, where the special case $\Omega_t = K+tL$,  where $t\in(0,\infty)$, was treated.

\begin{lemma} \label{lem:Minkowski-var}
Under the preceding regularity assumptions on $\{\Omega_t\}_{t\in(-\eta,\eta)}$, we have
\begin{equation} \label{eq:Minkowski-var}
\forall \ (t,y)\in(-\eta,\eta)\times\Sigma_0, \qquad \frac{\partial}{\partial t}F(t,y) = \psi_t(y) \hat\nu_{\Sigma_0}(y) + \mr{I\!I}_{\Sigma_0}^{-1}(y) \nabla_{\Sigma_0} \psi_t(y),
\end{equation}
where $\psi_t:\Sigma_0\to\R$ is given by $\psi_t(y) = \frac{\partial}{\partial t}g(t,y)$.
\end{lemma}

\begin{proof}
By the definition \eqref{eq:defn-g} of $g$ and the chain rule, we have
\begin{equation}
\nabla_{\Sigma_0}g(t,y) = \nabla_{\Sigma_0} \big( h_{\Omega_t}\circ \hat{\nu}_{\Sigma_0}\big)(y) = \nabla_{\Sigma_0} \hat{\nu}_{\Sigma_0}(y)  \nabla_{\mb{S}^{n-1}} h_{\Omega_t}\big( \hat{\nu}_{\Sigma_0}(y)\big) = \mr{I\!I}_{\Sigma_0}(y) \nabla_{\mb{S}^{n-1}} h_{\Omega_t}\big( \hat{\nu}_{\Sigma_0}(y)\big).
\end{equation}
On the other hand, for an arbitrary strictly convex body $L$,  we have (see, e.g., \cite[Proposition~6.3]{KM18})
\begin{equation}
\forall \ \theta\in\mb{S}^{n-1},\qquad \nabla_{\mb{S}^{n-1}} h_{L}(\theta) = \hat{\nu}_{\partial L}^{-1}(\theta) - \langle \hat{\nu}_{\partial L}^{-1}(\theta),\theta\rangle\theta = \hat{\nu}_{\partial L}^{-1}(\theta) - h_L(\theta)\theta.
\end{equation}
Therefore, combining the above, we derive
\begin{equation}
\nabla_{\Sigma_0}g(t,y) =  \mr{I\!I}_{\Sigma_0}(y) \big\{ \hat{\nu}_{\Sigma_t}^{-1} \circ\hat{\nu}_{\Sigma_0}(y) - h_{\Omega_t}\big(\hat{\nu}_{\Sigma_0}(y)\big) \hat{\nu}_{\Sigma_0}(y)\big\} = \mr{I\!I}_{\Sigma_0}(y) \big\{  F(t,y) - g(t,y) \hat{\nu}_{\Sigma_0}(y) \big\},
\end{equation}
which can be equivalently rewritten as
\begin{equation}
F(t,y) = g(t,y) \hat{\nu}_{\Sigma_0}(y) + \mr{I\!I}_{\Sigma_0}^{-1}(y) \nabla_{\Sigma_0}g(t,y). 
\end{equation}
Taking a derivative in $t$, we readily see that since $g$ is $C^2$,
\begin{equation}
\frac{\partial}{\partial t} F(t,y) = \frac{\partial }{\partial t} g(t,y) \hat{\nu}_{\Sigma_0}(y) + \mr{I\!I}_{\Sigma_0}^{-1}(y) \nabla_{\Sigma_0} \frac{\partial}{\partial t} g(t,y),
\end{equation}
which coincides with \eqref{eq:Minkowski-var} by the definition of $\psi_t$.
\end{proof}

Having computed the $t$-derivative of the parametrization $F_t:\Sigma_0\to\Sigma_t$, we can also derive formulas for the evolution of the measure of the convex sets $\{\Omega_t\}_{t\in(-\eta,\eta)}$.

\begin{lemma}
Let $\diff\rho(x) = e^{-U(x)}\,\diff x$ be an absolutely continuous finite measure on $\R^n$ with $C^1$ smooth density. Under the preceding regularity assumptions on $\{\Omega_t\}_{t\in(-\eta,\eta)}$, we have
\begin{equation} \label{eq:first-var}
\forall \ t\in(-\eta,\eta),\qquad \frac{\diff}{\diff t} \rho(\Omega_t)= \rmint_{\Sigma_t} \psi_t\big(F_t^{-1}(x)\big) \, \diff\rho(x) = \rmint_{\Sigma_0} \psi_t(y) e^{-U(F(t,y))} \mr{Jac} F_t(y)\, \diff y
\end{equation}
and
\begin{equation} \label{eq:second-var}
\frac{\diff^2}{\diff t^2}\bigg|_{t=0} \rho(\Omega_t) = \rmint_{\Sigma_0} \frac{\partial}{\partial t}\bigg|_{t=0}\psi_t \,\diff\rho +\rmint_{\Sigma_0} \msf{H}_{\Sigma_0,\rho} \cdot \psi_0^2 - \langle \mr{I\!I}_{\Sigma_0}^{-1} \nabla_{\Sigma_0} \psi_0, \nabla_{\Sigma_0}\psi_0\rangle\,\diff\rho.
\end{equation}
\end{lemma}

\begin{proof}
The first variation of $\rho(\Omega_t)$ only depends on the normal component of the derivative of the parametrization $F(t,\cdot)$ around the identity $F(0,\cdot)$, and thus
\begin{equation}
\frac{\diff}{\diff t} \bigg|_{t=0} \rho(\Omega_t) = \rmint_{\Sigma_0} \psi_0 \,\diff \rho,
\end{equation}
which coincides with \eqref{eq:first-var} at $t=0$.  To derive the formula in general, fix $T\in(-\eta,\eta)$ and notice that we also have
\begin{equation}
\forall \ (t,x)\in (-\eta,\eta)\times\Sigma_T,\qquad h_{\Omega_t} \big( \hat{\nu}_{\Sigma_T}(x)\big) = h_{\Omega_t}\big(\hat{\nu}_{\Sigma_0}\big(F_T^{-1}(x)\big)\big) = g\big(t, F_T^{-1}(x)\big).
\end{equation}
Therefore,  the same reasoning for the function $(t,x)\mapsto g(t,F_T^{-1}(x))$ around $t=T$ yields
\begin{equation*}
\frac{\diff}{\diff t}\bigg|_{t=T} \rho(\Omega_t) = \rmint_{\Sigma_T} \frac{\partial}{\partial t}\bigg|_{t=T} g\big(t,F_T^{-1}(x)\big)\, \diff \rho(x) = \rmint_{\Sigma_T} \psi_T\circ F_T^{-1}(x) \,\diff \rho(x)
\end{equation*}
and the last equality of \eqref{eq:first-var} follows by the change of variables $x=F_T(y)$. 

For the derivation of \eqref{eq:second-var}, we first use \eqref{eq:first-var} to get
\begin{equation}
\frac{\diff^2}{\diff t^2}\bigg|_{t=0} \rho(\Omega_t) = \rmint_{\Sigma_0} \frac{\partial}{\partial t}\bigg|_{t=0}\psi_t \,\diff\rho +\rmint_{\Sigma_0} \psi_0(y) \frac{\partial}{\partial t}\bigg|_{t=0} \Big\{ e^{-U(F(t,y))} \mr{Jac} F_t(y) \Big\} \,\diff y.
\end{equation}
The derivative of the latter quantity is computed in the proof of \cite[Theorem~6.6]{KM18} and plugging in this expression readily yields the desired identity \eqref{eq:second-var}.
\end{proof}


\subsection{The second derivative of marginals} Having presented the necessary background material, we can proceed to the proof of our main representation for the second derivative.

\begin{proposition} \label{prop:nguyen}
Fix $\diff\mu(x) = e^{-W(x)}\,\diff x$ a measure on $\R^{n}$, where $W:\R^n\to\R$ is $C^2$ smooth.  Let $\Omega \subseteq\R^{n+1}$ be a bounded open convex set with $C^3$ smooth boundary and for $t\in\R$, denote by $\Omega_t = \{x: \ (t,x)\in\Omega\}$ and $\Sigma_t = \partial\Omega_t$.  Assume that each $\Omega_t$ is strictly convex with nonempty interior for $t\in(-\eta,\eta)$.
Let $\Phi:\overline{\Omega}\to(0,\infty)$ be a function which is $C^{2}$ smooth and positive on $\overline{\Omega}$.  Moreover, for $\beta,\gamma\in\R\setminus\{0\}$,  consider the function $\varphi:(-\eta,\eta) \to\R_+$  given by
\begin{equation} \label{eq:phi}
\forall \ t\in (-\eta,\eta), \qquad \varphi(t) = \Bigg( \rmint_{\Omega_t} \Phi(t,x)^\beta\,\diff\mu(x)\Bigg)^{\gamma},
\end{equation}
Consider the probability measure $\nu$ on $\Omega_0$ given by 
\begin{equation}\label{eq:u}
\forall \ x\in \Omega_0,\qquad \frac{\diff\nu(x)}{\diff\mu(x)} = \frac{\Phi(0,x)^\beta}{\varphi(0)^{1/\gamma}}
\end{equation}
and let $u:\overline{\Omega_0}\to\R$ be a smooth function satisfying
\begin{equation} \label{eq:nguyen-eq}
\forall \ x\in \Omega_0,\qquad \ms{L}_{\nu} u(x) = \frac{\partial_t \Phi(0,x)}{\Phi(0,x)} - \rmint_{\Omega_0} \frac{\partial_t\Phi(0,\cdot)}{\Phi(0,\cdot)} \,\diff\nu - \frac{1}{\beta} \rmint_{\Sigma_0} \psi \,\diff\nu,
\end{equation}
and 
\begin{equation} \label{eq:nguyen-neumann}
\forall \ x\in\Sigma_0,\qquad \frac{\partial u(x)}{\partial \hat{\nu}_{\Sigma_0}} = -\frac{1}{\beta} \psi(x),
\end{equation}
where $\psi\equiv\psi_0:\Sigma_0\to\R$ is the variation appearing in \eqref{eq:Minkowski-var}. Then, we have
\begin{equation} \label{eq:nguyen}
\begin{split}
& \frac{1}{\gamma}\frac{\varphi''(0)}{\varphi(0)}  =  \beta \rmint_{\Omega_0} \frac{\langle\nabla^2_{(t,x)}\Phi(0,\cdot) X,X\rangle}{\Phi(0,\cdot)} \,\diff\nu - \beta^2 \rmint_{\Omega_0} \|\nabla^2 u\|_{\mr{HS}}^2 + \langle \nabla^2W \nabla u, \nabla u\rangle \, \diff\nu 
\\&  - \beta \rmint_{\Omega_0} (\ms{L}_\mu u)^2\,\diff\nu  - \beta(1-\beta\gamma) \Bigg(\rmint_{\Omega_0} \frac{\partial_t \Phi(0,\cdot)}{\Phi(0,\cdot)} \,\diff\nu\Bigg)^2-\Big(\frac{1}{\beta}-\gamma\Big) \Bigg(\rmint_{\Sigma_0} \psi \,\diff\nu\Bigg)^2
\\ & - 2 \beta \Bigg(\rmint_{\Omega_0} \ms{L}_\mu u\,\diff\nu\Bigg) \Bigg( \rmint_{\Omega_0} \frac{\partial_t \Phi(0,\cdot)}{\Phi(0,\cdot)}\,\diff\nu\Bigg) - 2 \Bigg(\rmint_{\Omega_0} \ms{L}_\mu u\,\diff\nu\Bigg) \Bigg(\rmint_{\Sigma_0} \psi \,\diff\nu\Bigg)
\\ & -2\beta\Big(\frac{1}{\beta}-\gamma\Big)  \Bigg( \rmint_{\Omega_0} \frac{\partial_t \Phi(0,\cdot)}{\Phi(0,\cdot)}\,\diff\nu\Bigg)\Bigg(\rmint_{\Sigma_0} \psi \,\diff\nu\Bigg) - \beta\rmint_{\Sigma_0} \frac{\partial_t\Phi(0,\cdot)}{\Phi(0,\cdot)} \cdot \psi\,\diff\nu + \rmint_{\Sigma_0} \frac{\partial}{\partial t}\bigg|_{t=0}\psi_t \,\diff\nu
\\ & - \rmint_{\Sigma_0} \beta^2 \langle \mr{I\!I}_{\Sigma_0}\nabla_{\Sigma_0} u,\nabla_{\Sigma_0} u\rangle +\langle \mr{I\!I}_{\Sigma_0}^{-1}\nabla_{\Sigma_0}\psi,\nabla_{\Sigma_0}\psi\rangle + 2\beta \langle \nabla_{\Sigma_0}\psi, \nabla_{\Sigma_0}u\rangle \,\diff\nu,
\end{split}
\end{equation}
where $X$ is the vector field $(1,-\beta\nabla u)$ on $\Omega_0$.
\end{proposition}

\begin{proof}
We shall make a slight abuse of notation by writing $\Phi_0\equiv\Phi(0,\cdot)$, $\partial_t\Phi_0 \equiv \partial_t \Phi(0,\cdot)$ and $\psi\equiv\psi_0$.  Observe that as the boundary $\Sigma$ of $\Omega$ is $C^3$ smooth, the same holds for each $\Sigma_t$.  Additionally, we will prove that the function $g:(-\eta,\eta)\times\Sigma_0\to\R$ defined by \eqref{eq:defn-g} is jointly  $C^2$ smooth.  To see this, observe that \eqref{eq:defn-g} and \eqref{eq:defn-F} readily imply that
\begin{equation}
\forall  \ (t,y)\in(-\eta,\eta)\times \Sigma_0, \qquad g(t,y) = \langle F(t,y) , \hat{\nu}_{\Sigma_0}(y) \rangle,
\end{equation}
so it suffices to prove that $(t,y)\mapsto (t,F_t(y))$ is $C^2$.  By the inverse function theorem, we equivalently need to prove that $\Sigma\cap\{|t|<\eta\} \ni (t,x) \mapsto (t,F_t^{-1}(x))$ is $C^2$ smooth.  
This is indeed true since
\begin{equation}
\forall \ (t,x)\in \Sigma\cap\{|t|<\eta\},\qquad F_t^{-1}(x) = \hat{\nu}_{\Sigma_0}^{-1}\circ \hat{\nu}_{\Sigma_t}(x)
\end{equation}
and
\begin{equation} \label{eq:norm-sec}
\forall \ x\in\Sigma_t, \qquad \hat{\nu}_{\Sigma_t}(x) = \frac{\hat{\nu}_\Sigma(t,x) - \langle \hat{\nu}_\Sigma(t,x), e_{1} \rangle e_1}{|\hat{\nu}_\Sigma(t,x) - \langle \hat{\nu}_\Sigma(t,x), e_{1} \rangle e_1|},
\end{equation}
and $\hat{\nu}_\Sigma$ is jointly $C^2$ smooth.  To justify the last identity, observe that the condition $\Omega_t\neq\emptyset$ for all $t\in(-\eta,\eta)$ implies that the denominator in \eqref{eq:norm-sec} never vanishes. 

Therefore, using Lemma \ref{lem:Minkowski-var} and the product rule, we differentiate under the integral sign to get
\begin{equation}
\frac{1}{\gamma}\varphi'(t) \stackrel{\eqref{eq:first-var}}{=}  \Bigg( \rmint_{\Omega_t} \Phi(t,x)^\beta\,\diff\mu(x)\Bigg)^{\gamma-1}\!\!\! \cdot \Bigg\{ \beta \rmint_{\Omega_t} \Phi(t,\cdot)^{\beta-1}\cdot \partial_t\Phi(t,\cdot)\,\diff\mu + \rmint_{\Sigma_t} \Phi(t,\cdot)^\beta\cdot \psi_t\circ F_t^{-1} \,\diff\mu\Bigg\}
\end{equation}
and similarly, since $F_0$ is the identity,
\begin{equation}
\begin{split}
\frac{1}{\gamma}& \varphi''(0) = (\gamma-1)\Bigg( \rmint_{\Omega_0} \Phi_0^\beta\,\diff\mu\Bigg)^{\gamma-2} \cdot \Bigg\{ \beta \rmint_{\Omega_0} \Phi_0^{\beta-1}\cdot \partial_t\Phi_0\,\diff\mu + \rmint_{\Sigma_0} \Phi_0^\beta\cdot \psi \,\diff\mu\Bigg\}^2
\\ & + \Bigg( \rmint_{\Omega_0} \Phi_0^\beta\,\diff\mu\Bigg)^{\gamma-1} \cdot \frac{\diff }{\diff t} \bigg|_{t=0} \Bigg\{ \beta \rmint_{\Omega_t} \Phi(t,\cdot)^{\beta-1} \cdot \partial_t\Phi(t,\cdot)\,\diff\mu + \rmint_{\Sigma_t} \Phi(t,\cdot)^\beta\cdot \psi_t\circ F_t^{-1} \,\diff\mu\Bigg\}.
\end{split}
\end{equation}
To evaluate the latter derivative, observe that \eqref{eq:first-var} again yields
\begin{equation}
\begin{split}
 \frac{\diff }{\diff t} \bigg|_{t=0}  \beta \rmint_{\Omega_t} \Phi(t,\cdot)^{\beta-1}\cdot \partial_t\Phi(t,\cdot)\,\diff\mu \stackrel{\eqref{eq:first-var}}{=}  \beta  \rmint_{\Omega_0}(\beta-1)\cdot \Phi_0^{\beta-2} \cdot (\partial_t&\Phi_0) ^2 + \Phi_0^{\beta-1}\cdot \partial_{tt}\Phi_0 \,\diff\mu 
 \\ & + \beta \rmint_{\Sigma_0} \Phi_0^{\beta-1}\cdot \partial_t \Phi_0\cdot \psi \,\diff\mu,
 \end{split}
\end{equation}
whereas \eqref{eq:second-var} combined with a change of variables gives
\begin{equation}
\begin{split}
\frac{\diff }{\diff t} \bigg|_{t=0}  \rmint_{\Sigma_t} \Phi(t,\cdot)^\beta  \cdot \psi_t\circ F_t^{-1} &\,\diff\mu  = \frac{\diff }{\diff t} \bigg|_{t=0} \rmint_{\Sigma_0} \psi_t(y) \cdot e^{-W(F(t,y))}\cdot \Phi(t,F(t,y))^\beta \cdot \mathrm{Jac}F_t(y) \, \diff y
\\ & \stackrel{\eqref{eq:second-var}}{=} \rmint_{\Sigma_0} \frac{\partial}{\partial t}\bigg|_{t=0}\psi_t \,\diff\widetilde{\nu} + \rmint_{\Sigma_0} \msf{H}_{\Sigma_0,\widetilde{\nu}}\cdot  \psi^2 - \langle \mr{I\!I}_{\Sigma_0}^{-1}\nabla_{\Sigma_0}\psi,\nabla_{\Sigma_0}\psi\rangle \,\diff\widetilde{\nu},
\end{split}
\end{equation}
where we use the ad hoc notation $\diff\widetilde{\nu} \eqdef \Phi_0^\beta\, \diff\mu$. Considering the normalized measure $\nu\equiv\frac{\widetilde{\nu}}{\widetilde{\nu}(\Omega_0)}$ and combining all the above, we derive the formula
\begin{equation} \label{eq:second-derivative}
\begin{split}
\frac{1}{\gamma} \frac{\varphi''(0)}{\varphi(0)}
& =\beta \rmint_{\Omega_0} \frac{\partial_{tt}\Phi_0}{\Phi_0}\,\diff\nu+  \beta(\beta-1)\rmint_{\Omega_0} \left(\frac{\partial_t\Phi_0}{\Phi_0}\right)^2\,\diff\nu +\beta^2(\gamma-1) \Bigg(\rmint_{\Omega_0} \frac{\partial_t \Phi_0}{\Phi_0} \,\diff\nu\Bigg)^2
\\ & +2\beta(\gamma-1) \Bigg(\rmint_{\Omega_0} \frac{\partial_t \Phi_0}{\Phi_0} \,\diff\nu \Bigg)\Bigg( \rmint_{\Sigma_0} \psi \,\diff\nu \Bigg) + (\gamma-1) \Bigg(\rmint_{\Sigma_0} \psi\,\diff\nu\Bigg)^2
\\ & + \beta\rmint_{\Sigma_0} \frac{\partial_t \Phi_0}{\Phi_0} \cdot \psi\,\diff\nu + \rmint_{\Sigma_0} \frac{\partial}{\partial t}\bigg|_{t=0}\psi_t \,\diff\nu+ \rmint_{\Sigma_0} \msf{H}_{\Sigma_0,\nu}\cdot \psi^2 - \langle \mr{I\!I}_{\Sigma_0}^{-1}\nabla_{\Sigma_0}\psi,\nabla_{\Sigma_0}\psi\rangle \,\diff \nu
\\ & = \beta \rmint_{\Omega_0} \frac{\partial_{tt}\Phi_0}{\Phi_0}\,\diff\nu + \beta(\beta-1)\mathrm{Var}_{\nu} \Bigg( \frac{\partial_t\Phi_0}{\Phi_0} \Bigg) +\beta(\beta\gamma-1) \Bigg(\rmint_{\Omega_0} \frac{\partial_t \Phi_0}{\Phi_0} \,\diff\nu\Bigg)^2
\\ & +2\beta(\gamma-1) \Bigg(\rmint_{\Omega_0} \frac{\partial_t \Phi_0}{\Phi_0} \,\diff\nu \Bigg)\Bigg( \rmint_{\Sigma_0} \psi \,\diff\nu \Bigg) + (\gamma-1) \Bigg(\rmint_{\Sigma_0} \psi\,\diff\nu\Bigg)^2
\\ & + \beta\rmint_{\Sigma_0} \frac{\partial_t \Phi_0}{\Phi_0} \cdot \psi\,\diff\nu  + \rmint_{\Sigma_0} \frac{\partial}{\partial t}\bigg|_{t=0}\psi_t \,\diff\nu+ \rmint_{\Sigma_0} \msf{H}_{\Sigma_0,\nu}\cdot \psi^2 - \langle \mr{I\!I}_{\Sigma_0}^{-1}\nabla_{\Sigma_0}\psi,\nabla_{\Sigma_0}\psi\rangle \,\diff \nu.
\end{split}
\end{equation}
Consider now the function $u:\overline{\Omega_0}\to\R$ solving equation \eqref{eq:nguyen-eq} with Neumann boundary data \eqref{eq:nguyen-neumann}. By the integration by parts formula \eqref{eq:bp}, we can rewrite the variance as
\begin{equation}
\begin{split}
&\mathrm{Var}_{\nu} \Bigg( \frac{\partial_t\Phi_0}{\Phi_0} \Bigg) = \rmint_{\Omega_0} \Bigg( \frac{\partial_t\Phi_0}{\Phi_0} - \rmint_{\Omega_0} \frac{\partial_t\Phi_0}{\Phi_0}\,\diff\nu  \Bigg) \ms{L}_\nu u \,\diff \nu
\\ & \stackrel{\eqref{eq:bp}\wedge\eqref{eq:nguyen-neumann}}{=} - \rmint_{\Omega_0} \left\langle \nabla_x \Bigg(\frac{\partial_t\Phi_0}{\Phi_0} \Bigg), \nabla u \right\rangle \,\diff\nu -\frac{1}{\beta} \rmint_{\Sigma_0} \Bigg( \frac{\partial_t\Phi_0}{\Phi_0} - \rmint_{\Omega_0} \frac{\partial_t\Phi_0}{\Phi_0}\,\diff\nu  \Bigg) \psi\,\diff\nu
\\ & = - \rmint_{\Omega_0} \left\langle \nabla_x \Bigg(\frac{\partial_t\Phi_0}{\Phi_0} \Bigg), \nabla u \right\rangle \,\diff\nu -\frac{1}{\beta} \rmint_{\Sigma_0}\frac{\partial_t\Phi_0}{\Phi_0} \cdot\psi\,\diff\nu + \frac{1}{\beta} \Bigg( \rmint_{\Omega_0} \frac{\partial_t\Phi_0}{\Phi_0}\,\diff\nu \Bigg) \Bigg(\rmint_{\Sigma_0}\psi\,\diff\nu\Bigg).
\end{split}
\end{equation}
Similarly, using again the equation \eqref{eq:nguyen-eq} and the Bochner--Reilly formula \eqref{eq:bochner}, we get
\begin{equation}
\begin{split}
\mathrm{Var}_{\nu} &\Bigg( \frac{\partial_t\Phi_0}{\Phi_0} \Bigg) \stackrel{\eqref{eq:nguyen-eq}}{=} \mathrm{Var}_{\nu} \big( \ms{L}_\nu u \big)  = \rmint_{\Omega_0} \big(\ms{L}_\nu u\big)^2\,\diff\nu - \Bigg( \rmint_{\Omega_0} \ms{L}_\nu u \,\diff\nu \Bigg)^2  
\\ & \stackrel{\eqref{eq:nguyen-eq}}{=}\rmint_{\Omega_0} \big(\ms{L}_\nu u\big)^2\,\diff\nu - \frac{1}{\beta^2} \Bigg(\rmint_{\Sigma_0} \psi\,\diff\nu\Bigg)^2
\\ & \stackrel{\eqref{eq:nguyen-neumann}\wedge\eqref{eq:bochner}}{=} \rmint_{\Omega_0} \|\nabla^2u\|_{\mr{HS}}^2 +  \langle \nabla^2W\nabla u, \nabla u \rangle\!-\! \beta\left\langle\frac{ \nabla^2_x\Phi_0}{\Phi_0} \nabla u,\nabla u\right\rangle + \beta \frac{\langle \nabla_x\Phi_0,\nabla u\rangle^2}{\Phi_0^2}\,\diff\nu
\\ & +\rmint_{\Sigma_0} \frac{1}{\beta^2} \msf{H}_{\Sigma_0,\nu}\cdot \psi^2 + \langle \mr{I\!I}_{\Sigma_0}\nabla_{\Sigma_0}u,\nabla_{\Sigma_0}u\rangle +\frac{2}{\beta}\langle \nabla_{\Sigma_0}\psi,\nabla_{\Sigma_0}u\rangle \,\diff\nu- \frac{1}{\beta^2} \Bigg(\rmint_{\Sigma_0} \psi\,\diff\nu\Bigg)^2.
\end{split}
\end{equation}
Therefore, combining the above, we get 
\begin{equation} \label{eq:need-to-replace}
\begin{split}
&\mathrm{Var}_{\nu} \Bigg( \frac{\partial_t\Phi_0}{\Phi_0} \Bigg) = - \rmint_{\Omega_0} \|\nabla^2u\|_{\mr{HS}}^2 +  \langle \nabla^2W\nabla u, \nabla u \rangle\!-\! \beta\left\langle\frac{ \nabla^2_x\Phi_0}{\Phi_0} \nabla u,\nabla u\right\rangle + \beta \underbrace{\frac{\langle \nabla_x\Phi_0,\nabla u\rangle^2}{\Phi_0^2}}_{\raisebox{.5pt}{\textcircled{\raisebox{-.9pt} {I}}} }\,\diff\nu
\\ & -2 \rmint_{\Omega_0} \left\langle \nabla_x \Bigg(\frac{\partial_t\Phi_0}{\Phi_0} \Bigg), \nabla u \right\rangle \,\diff\nu -\frac{2}{\beta} \rmint_{\Sigma_0}\frac{\partial_t\Phi_0}{\Phi_0} \cdot\psi\,\diff\nu + \frac{2}{\beta} \Bigg( \rmint_{\Omega_0} \frac{\partial_t\Phi_0}{\Phi_0}\,\diff\nu \Bigg) \Bigg(\rmint_{\Sigma_0}\psi\,\diff\nu\Bigg)
\\ & - \rmint_{\Sigma_0} \frac{1}{\beta^2} \msf{H}_{\Sigma_0,\nu}\cdot\psi^2 + \langle \mr{I\!I}_{\Sigma_0}\nabla_{\Sigma_0}u,\nabla_{\Sigma_0}u\rangle +\frac{2}{\beta}\langle \nabla_{\Sigma_0}\psi,\nabla_{\Sigma_0}u\rangle \,\diff\nu+ \frac{1}{\beta^2} \Bigg(\rmint_{\Sigma_0} \psi\,\diff\nu\Bigg)^2.
\end{split}
\end{equation}
We shall now use equations \eqref{eq:nguyen-eq} and \eqref{eq:nguyen-neumann} to rewrite the integral of the term $\raisebox{.5pt}{\textcircled{\raisebox{-.9pt} {I}}}$. Indeed, notice that by the definitions of $\ms{L}_\mu$, $\ms{L}_{\nu}$ and  \eqref{eq:nguyen-eq}, we have
\begin{equation} \label{eq:I}
\begin{split}
\beta^2\rmint_{\Omega_0}\raisebox{.5pt}{\textcircled{\raisebox{-.9pt} {I}}}&\,\diff\nu  = \beta^2 \rmint_{\Omega_0} \frac{\langle \nabla_x\Phi_0,\nabla u\rangle^2}{\Phi_0^2}\,\diff\nu = \rmint_{\Omega_0} \big(\ms{L}_{\nu} u- \ms{L}_\mu u\big)^2\,\diff\nu 
\\ &= \rmint_{\Omega_0} \big(\ms{L}_\mu u\big)^2 +2 \ms{L}_\nu u\cdot (\ms{L}_\nu u -\ms{L}_\mu u )\,\diff\nu - \rmint_{\Omega_0} \big( \ms{L}_\nu u\big)^2\,\diff\nu
\\ & \stackrel{\eqref{eq:nguyen-eq}}{=} \rmint_{\Omega_0} \big(\ms{L}_\mu u\big)^2 + 2\beta\ms{L}_{\nu} u \cdot  \left\langle \frac{\nabla_x\Phi_0}{\Phi_0}, \nabla u\right\rangle \,\diff\nu - \mathrm{Var}_{\nu} \Bigg( \frac{\partial_t\Phi_0}{\Phi_0} \Bigg) - \frac{1}{\beta^2} \Bigg(\rmint_{\Sigma_0} \psi\,\diff\nu\Bigg)^2.
\end{split}
\end{equation}
Therefore, \eqref{eq:need-to-replace} can be rewritten as
\begin{equation*}
\begin{split}
(\beta-1)\mathrm{Var}_{\nu} \Bigg( \frac{\partial_t\Phi_0}{\Phi_0} \Bigg) & = - \beta\rmint_{\Omega_0}\!\! \|\nabla^2u\|_{\mr{HS}}^2 +  \langle \nabla^2W\nabla u, \nabla u \rangle\!-\! \beta\left\langle\frac{ \nabla^2_x\Phi_0}{\Phi_0} \nabla u,\nabla u\right\rangle\,\diff\nu - \!\rmint_{\Omega_0}\!\!\!\big( \ms{L}_\mu u\big)^2\,\diff\nu
\\ & - 2\beta \rmint_{\Omega_0}\underbrace{\left\langle \nabla_x \Bigg(\frac{\partial_t\Phi_0}{\Phi_0} \Bigg), \nabla u \right\rangle}_{\raisebox{.5pt}{\textcircled{\raisebox{-.9pt} {I\!I}}}} + \ms{L}_{\nu} u \cdot  \left\langle \frac{\nabla_x\Phi_0}{\Phi_0}, \nabla u\right\rangle \,\diff\nu  -2 \rmint_{\Sigma_0}\frac{\partial_t\Phi_0}{\Phi_0} \cdot\psi\,\diff\nu
\\ & + 2 \Bigg( \rmint_{\Omega_0} \frac{\partial_t\Phi_0}{\Phi_0}\,\diff\nu \Bigg) \Bigg(\rmint_{\Sigma_0}\psi\,\diff\nu\Bigg) + \frac{\beta+1}{\beta^2} \Bigg(\rmint_{\Sigma_0} \psi\,\diff\nu\Bigg)^2
\\ & - \rmint_{\Sigma_0} \frac{1}{\beta} \msf{H}_{\Sigma_0,\nu}\cdot \psi^2 + \beta \langle \mr{I\!I}_{\Sigma_0}\nabla_{\Sigma_0}u,\nabla_{\Sigma_0}u\rangle +2\langle \nabla_{\Sigma_0}\psi,\nabla_{\Sigma_0}u\rangle \,\diff\nu.
\end{split}
\end{equation*}
For term $\raisebox{.5pt}{\textcircled{\raisebox{-.9pt} {I\!I}}}$,  employing again \eqref{eq:nguyen-eq}, we can write
    \begin{alignat*}{2}
\beta\rmint_{\Omega_0} \raisebox{.5pt}{\textcircled{\raisebox{-.9pt} {I\!I}}}\,\diff\nu & =\beta \rmint_{\Omega_0} \left\langle \frac{ \nabla_x \partial_t \Phi_0 }{\Phi_0},\nabla u \right\rangle  && - \frac{\partial_t\Phi_0}{\Phi_0} \left\langle \frac{\nabla_x \Phi_0}{\Phi_0}, \nabla u\right\rangle \,\diff\nu
\\ & \stackrel{\eqref{eq:nguyen-eq}}{=} \beta \rmint_{\Omega_0}\left\langle \frac{ \nabla_x \partial_t \Phi_0 }{\Phi_0},\nabla u \right\rangle  && -\ms{L}_{\nu} u\cdot \left\langle \frac{\nabla_x \Phi_0}{\Phi_0}, \nabla u\right\rangle \,\diff\nu 
\\ & &&- \beta \Bigg( \rmint_{\Omega_0}\frac{\partial_t\Phi_0}{\Phi_0}\,\diff\nu+\frac{1}{\beta} \rmint_{\Sigma_0} \psi\,\diff\nu \Bigg) \Bigg( \rmint_{\Omega_0}\left\langle \frac{\nabla_x \Phi_0}{\Phi_0}, \nabla u\right\rangle\,\diff\nu \Bigg)
\\ & = \beta \rmint_{\Omega_0}\left\langle \frac{ \nabla_x \partial_t \Phi_0 }{\Phi_0},\nabla u \right\rangle  && -\ms{L}_{\nu} u\cdot \left\langle \frac{\nabla_x \Phi_0}{\Phi_0}, \nabla u\right\rangle \,\diff\nu 
\\ & &&+ \Bigg( \rmint_{\Omega_0}\frac{\partial_t\Phi_0}{\Phi_0}\,\diff\nu+\frac{1}{\beta} \rmint_{\Sigma_0} \psi\,\diff\nu \Bigg) \Bigg( \rmint_{\Omega_0} \ms{L}_\mu u \,\diff\nu + \frac{1}{\beta} \rmint_{\Sigma_0} \psi\,\diff\nu \Bigg),
\end{alignat*}
where in the last equality we used again that $\ms{L}_{\nu}u-\ms{L}_\mu u = \beta \left\langle \frac{\nabla_x \Phi_0}{\Phi_0},\nabla u\right\rangle$ and integrated with respect to $\nu$. Putting everything together, we finally derive the variance representation
\begin{equation*}
\begin{split}
(\beta-1)\mathrm{Var}_{\nu}& \Bigg( \frac{\partial_t\Phi_0}{\Phi_0} \Bigg)  = - \beta\rmint_{\Omega_0}\!\! \|\nabla^2u\|_{\mr{HS}}^2 +  \langle \nabla^2W\nabla u, \nabla u \rangle\!-\! \beta\left\langle\frac{ \nabla^2_x\Phi_0}{\Phi_0} \nabla u,\nabla u\right\rangle+2\left\langle \frac{ \nabla_x \partial_t \Phi_0 }{\Phi_0},\nabla u \right\rangle\,\diff\nu 
\\ &- \rmint_{\Omega_0}\big( \ms{L}_\mu u\big)^2\,\diff\nu - 2  \Bigg( \rmint_{\Omega_0}\frac{\partial_t\Phi_0}{\Phi_0}\,\diff\nu+\frac{1}{\beta} \rmint_{\Sigma_0} \psi\,\diff\nu \Bigg) \Bigg( \rmint_{\Omega_0} \ms{L}_\mu u \,\diff\nu + \frac{1}{\beta} \rmint_{\Sigma_0} \psi\,\diff\nu \Bigg)
\\ & - 2 \rmint_{\Sigma_0}\frac{\partial_t\Phi_0}{\Phi_0} \cdot\psi\,\diff\nu + 2 \Bigg( \rmint_{\Omega_0} \frac{\partial_t\Phi_0}{\Phi_0}\,\diff\nu \Bigg) \Bigg(\rmint_{\Sigma_0}\psi\,\diff\nu\Bigg) + \frac{\beta+1}{\beta^2} \Bigg(\rmint_{\Sigma_0} \psi\,\diff\nu\Bigg)^2
\\ & - \rmint_{\Sigma_0} \frac{1}{\beta} \msf{H}_{\Sigma_0,\nu}\cdot \psi^2 + \beta \langle \mr{I\!I}_{\Sigma_0}\nabla_{\Sigma_0}u,\nabla_{\Sigma_0}u\rangle +2\langle \nabla_{\Sigma_0}\psi,\nabla_{\Sigma_0}u\rangle \,\diff\nu.
\end{split}
\end{equation*}
Finally, substituting this formula for the variance into \eqref{eq:second-derivative} and simplifying, we conclude that 
\begin{equation}
\begin{split}
 \frac{1}{\gamma} & \frac{\varphi''(0)}{\varphi(0)}
 = \beta \rmint_{\Omega_0} \frac{\partial_{tt}\Phi_0}{\Phi_0} - 2\beta\left\langle \frac{\nabla_x\partial_t\Phi_0}{\Phi_0}, \nabla u \right\rangle + \beta^2\left\langle\frac{ \nabla^2_x\Phi_0}{\Phi_0} \nabla u,\nabla u\right\rangle \,\diff\nu
\\ & -\beta^2\rmint_{\Omega_0} \|\nabla^2u\|_{\mr{HS}}^2 +  \langle \nabla^2W\nabla u, \nabla u \rangle\,\diff\nu 
\\ & - \beta \rmint_{\Omega_0} \big(\ms{L}_\mu u\big)^2\,\diff\nu - \beta(1-\beta\gamma) \Bigg(\rmint_{\Omega_0} \frac{\partial_t \Phi_0}{\Phi_0} \,\diff\nu\Bigg)^2-\Big(\frac{1}{\beta}-\gamma\Big) \Bigg(\rmint_{\Sigma_0} \psi \,\diff\nu\Bigg)^2
\\ & - 2 \beta \Bigg(\rmint_{\Omega_0} \ms{L}_\mu u\,\diff\nu\Bigg) \Bigg( \rmint_{\Omega_0} \frac{\partial_t \Phi_0}{\Phi_0}\,\diff\nu\Bigg) - 2 \Bigg(\rmint_{\Omega_0} \ms{L}_\mu u\,\diff\nu\Bigg) \Bigg(\rmint_{\Sigma_0} \psi \,\diff\nu\Bigg)
\\ & -2\beta\Big(\frac{1}{\beta}-\gamma\Big)  \Bigg( \rmint_{\Omega_0} \frac{\partial_t \Phi_0}{\Phi_0}\,\diff\nu\Bigg)\Bigg(\rmint_{\Sigma_0} \psi \,\diff\nu\Bigg) - \beta\rmint_{\Sigma_0} \frac{\partial_t\Phi_0}{\Phi_0} \cdot \psi\,\diff\nu + \rmint_{\Sigma_0} \frac{\partial}{\partial t}\bigg|_{t=0}\psi_t \,\diff\nu
\\ & - \rmint_{\Sigma_0} \beta^2 \langle \mr{I\!I}_{\Sigma_0}\nabla_{\Sigma_0} u,\nabla_{\Sigma_0} u\rangle +\langle \mr{I\!I}_{\Sigma_0}^{-1}\nabla_{\Sigma_0}\psi,\nabla_{\Sigma_0}\psi\rangle + 2\beta \langle \nabla_{\Sigma_0}\psi, \nabla_{\Sigma_0}u\rangle \,\diff\nu.
\end{split}
\end{equation}
The proof of Proposition \ref{prop:nguyen} is now complete if one observes that
\begin{equation*}
  \langle\nabla^2_{(t,x)}\Phi(0,x) X(x),X(x)\rangle=\partial_{tt}\Phi_0(x) - 2\beta\left\langle \nabla_x\partial_t\Phi_0(x), \nabla u(x) \right\rangle + \beta^2\left\langle\nabla^2_x\Phi_0(x)\nabla u(x),\nabla u(x)\right\rangle
\end{equation*}
for the vector field $X(x)=(1,-\beta\nabla u(x))$ on $\Omega_0$.
\end{proof}


\section{A curvature condition and the proof of Theorem \ref{thm:bbl}} \label{sec:main}

\subsection{Hereditary convexity} We start by emphasizing the key \emph{curvature} property which is sufficient for our results. Variants of this have already appeared several times in the literature, e.g.~in  \cite{EM21,Liv24,CR23,AR23}. As we shall see, radially symmetric log-concave measures will satisfy this property. 

\begin{definition}\label{def:curvature}
Let $\mu$ be an even measure  on $\R^n$ with density $e^{-W}:\R^n\to\R_+$, where $W$ is a $C^2$ smooth function. We say that $\mu$ is \emph{hereditarily convex} (with respect to even measures) if 
\begin{itemize}
    \item[(i)] $\forall \ z\in \R^n$, we have $\langle \nabla W (z) , z \rangle \ge 0$;
    \item[(ii)] For every \emph{even} probability measure $\nu$ which is \emph{log-concave with
respect to $\mu$},  and for every \emph{even} smooth function $u$ with $\ms{L}_\mu(u)\in L_2(\nu)$,  we have the inequality
\begin{equation}\label{eq:x2}
\rmint_{\R^n} \|\nabla^2u\|_{\mr{HS}}^2 + \langle \nabla^2 W \nabla u, \nabla u\rangle \,\diff\nu \geq 
\frac{\Big( \rmint_{\R^n} \ms{L}_\mu u\,\diff\nu\Big)^2 }
{\rmint_{\R^n} \ms{L}_\mu (|x|^2/2) \,\diff\nu}.
\end{equation}
In this statement, we allow $\nu$ to have a smooth support $U\subsetneq \R^n$, in which case we ask that~\eqref{eq:x2} is satisfied for smooth $u:\overline{U}\to\R$ with arbitrary Neumann data on $\partial U$.
\end{itemize}

\end{definition}

Note that the first condition is automatically satisfied when $\mu$ is log-concave. 
We will see below that  the denominator in the right hand-side of~\eqref{eq:x2} is positive when $\nu$ is not proportional to $\mu$.  In the degenerate scenario in which the two measures are proportional,   the ratio is simply interpreted as 0.  In any case, one can replace the right-hand side of~\eqref{eq:x2} by the supremum over $\lambda\in \R$ of
$$2\lambda \rmint_{\R^n} \ms{L}_\mu u \, \diff\nu - \lambda^2 \rmint_{\R^n} \ms{L}_\mu (|x|^2/2) \, \diff\nu.$$
Actually, as we shall see, it would also make sense to study the same property for fixed $\lambda$ in place of the supremum. It is interesting to note that the Lebesgue measure is hereditarily convex without the restriction of $\nu$ and $u$ being even: in that case $W=0$ and $\ms{L}_\mu u= \Delta u$, so~\eqref{eq:x2} holds for any finite measure $\nu$ and every smooth function $u$, in view of the pointwise inequality $\|\nabla^2 u \|^2_{\mr{HS}} \ge \frac1n (\Delta u)^2$. 

It is important to make some further observations on the right hand-side of~\eqref{eq:x2}. Write
$$\diff\nu = {\bf 1}_U \,  e^{-\Psi}\, \diff\mu$$
where $U$ is a smooth (open) convex set and assume that $\Psi$ is $C^1$ on $\overline U$. 
The denominator of \eqref{eq:x2} then satisfies the inequalities
\begin{equation}
\begin{split}
n & \geq n-\rmint_U \langle \nabla W(x),x\rangle \,\diff\nu(x)=  \rmint_U \ms{L}_\mu (|x|^2/2) \, \diff\nu
\\ & = \rmint_U \langle \nabla \Psi(x), x \rangle \,\diff\nu(x) + \rmint_{\partial U} \langle \nu_{\partial K}(x) , x \rangle \, \diff \nu(x)   \ge   \rmint_U \langle \nabla \Psi(x), x \rangle \,\diff\nu(x) \ge  0
\end{split}
\end{equation}
where we used property $(i)$, and the fact that convexity and symmetry imply that 
$ \langle \nabla \Psi(x), x \rangle\ge 0$ and $\langle \nu_K(x) , x \rangle \ge 0$. 
In particular,  we get that $\rmint_{U} \ms{L}_\mu(|x|^2/2) \,\diff\nu >0$ if $\nu$ is not proportional to $\mu$.  So we see that our condition implicitly incorporates the  dimension, in the form
\begin{equation}\label{eq:x2n}
\rmint_{\R^n} \|\nabla^2u\|_{\mr{HS}}^2 + \langle \nabla^2 W \nabla u, \nabla u\rangle \,\diff\nu \geq 
\frac1n \Bigg( \rmint \ms{L}_\mu u\,\diff\nu\Bigg)^2 .
\end{equation}

As we shall see, heriditary convexity is a sufficient condition for the open questions we raised in the introduction, in the sense that given an even measure $\mu$ with smooth density on $\R^n$,
\begin{eqnarray}\label{eq:prop_impl}
 & &    \textrm{$\mu$ verifies the dimensional BM-conjecture~\eqref{eq:dim-BM-conj}} \notag \\
 & \Nearrow &  \notag\\
 \textrm{$\mu$ is even hereditarily convex} & &  \\
 & \Searrow  & \notag \\
 & &  \textrm{$\mu$ verifies the B-conjecture~\eqref{eq:B-conj}}. \notag
\end{eqnarray}
For the questions under study, smoothness can be enforced by approximation, although it could be of independent interest to allow $\mu$ to have a support in the definition above. It should be said that there is some hidden differences between the implications~\eqref{eq:prop_impl}. The fact that~\eqref{eq:x2n}, a consequence of~\eqref{eq:x2}, restricted to measures $\nu$ with bounded support implies the Brunn--Minkowski inequality \eqref{eq:dim-BM-conj} is due to \cite{KM18,KL21}, yet soon we will prove that it implies the stronger functional form of Theorem \ref{thm:bbl}. On the contrary, to derive the B-conjecture from \eqref{eq:x2}, we will only consider measures $\nu$ supported on $\R^n$ having a smooth positive log-concave density with respect to $\mu$.

We now establish that rotationally invariant log-concave measures $\mu$ satisfy the hereditary convexity condition. This amounts to an optimization of an argument originating in \cite{EM21}, while incorporating a novel weighted Poincaré inequality from~\cite{CR23}, which remains at the heart of all recent developments exhibiting refined concavity properties of rotationally invariant measures. 

\begin{proposition} \label{prop:x2}
Let $w:[0,\infty)\to(-\infty,\infty)$ be a $C^2$ smooth increasing function such that $t\mapsto w(e^t)$ is convex on $\R$ and let $\mu$ be the measure on $\R^n$ with density $e^{-w(|x|)}$. 
Then, $\mu$ is heriditarily convex. 
\end{proposition}

\begin{proof}
Write $\diff \mu(x) = e^{-W(x)}\diff x$ with $W(x)=w(|x|)$. Since $w$ is increasing, we have 
$\langle \nabla W(x) , x \rangle \ge 0$, and thus condition $(i)$ from Definition~\ref{def:curvature} is verified. 
Let $\nu$ and $u$ be as in part $(ii)$ of Definition~\ref{def:curvature} and let $\lambda\in\R$ be a parameter to be chosen later. Consider the function $v:\R^n\to\R$ given by 
$$\forall \ x\in\R^n,\qquad v(x) = u(x)+\frac{\lambda}{2} |x|^2 .$$
Since $v$ is even, each partial derivative $\partial_i v$ is odd and thus \cite[Theorem~4]{CR23} yields
\begin{equation}\label{eq:CR23}
\begin{split}
\rmint_{\R^n}  \|\nabla^2v\|_{\mr{HS}}^2 \,\diff\nu & = \sum_{i=1}^n \rmint_{\R^n} \big|\nabla \partial_i v\big|^2\,\diff\nu(x) 
\\ & \geq \sum_{i=1}^n \rmint_{\R^n}\frac{w'(|x|)}{|x|} \big(\partial_i v(x)\big)^2\,\diff\nu(x) = \rmint_{\R^n}\frac{w'(|x|)}{|x|} |\nabla v(x)|^2\,\diff\nu(x)
\end{split}
\end{equation}
From the definition of $v$,  we have, pointwise,
\begin{equation*} \label{eq:split-grad}
|\nabla v(x)|^2 = |\nabla u(x)|^2 + 2\lambda \langle \nabla u(x), x\rangle + \lambda^2 |x|^2
\end{equation*}
and
\begin{equation*} \label{eq:split-x2}
\begin{split}
 \|\nabla^2&v\|_{\mr{HS}}^2 = \|\nabla^2u\|_{\mr{HS}}^2 + 2\lambda \Delta u +\lambda^2 n.
\end{split}
\end{equation*}
Plugging this in~\eqref{eq:CR23}, we get
\begin{eqnarray*}
\rmint_{\R^n}  \|\nabla^2u\|_{\mr{HS}}^2 \, d\diff\nu
&\ge &  \rmint_{\R^n}\frac{w'(|x|)}{|x|} |\nabla u(x)|^2\,\diff\nu(x)  
-2\lambda \rmint_{\R^n} \Big(\Delta u(x)  - \frac{w'(|x|)}{|x|} \langle \nabla u(x), x\rangle \Big)\,\diff\nu(x) \\ 
& & \quad   - \lambda^2\Big( n- \rmint_{\R^n} |x| w'(|x|) \,\diff\nu(x)\Big) \\
  &= & \rmint_{\R^n}\frac{w'(|x|)}{|x|} |\nabla u(x)|^2\,\diff\nu(x)  - 
 2\lambda \rmint_{\R^n} \ms{L}_\mu u \, \diff\nu - \lambda^2 \rmint_{\R^n} \ms{L}_\mu (|x|^2/2) \, \diff\nu,
\end{eqnarray*}
since $\ms{L}_\mu h = \Delta h -  \frac{w'(|x|)}{|x|} \langle \nabla h, x\rangle $ for any smooth function $h$, where the dependence on $x$ is implicit to simplify notation.
Therefore we have that 
\begin{multline*}\label{eq:x2-almost-done}
    \rmint_{\R^n} \|\nabla^2 u\|_{\mr{HS}}^2 + \langle \nabla^2 W \nabla u, \nabla u\rangle\,\diff\nu \\ 
    \geq\rmint_{\R^n} \left\langle \Big( \nabla^2W + \frac{w'(|x|)}{|x|} \msf{Id} \Big) \nabla u,\nabla u \right\rangle\,\diff\nu
 -  2\lambda \rmint_{\R^n} \ms{L}_\mu u \, \diff\nu - \lambda^2 \rmint_{\R^n} \ms{L}_\mu (|x|^2/2) \, \diff\nu.
\end{multline*} 
But  by our assumption that $t\mapsto w(e^t)$ is increasing and convex, we have $\nabla^2W(x) + \frac{w'(|x|)}{|x|} \msf{Id} \succeq 0$ and thus the first integral is nonnegative. Therefore, we  get
\begin{equation*}
 \rmint_{\R^n} \|\nabla^2 u\|_{\mr{HS}}^2 + \langle \nabla^2 W \nabla u, \nabla u\rangle\,\diff\nu \ge   - 
 2\lambda \rmint_{\R^n} \ms{L}_\mu u \, \diff\nu - \lambda^2 \rmint_{\R^n} \ms{L}_\mu (|x|^2/2) \, \diff\nu.
\end{equation*}
Optimizing over $\lambda$ gives the desired inequality \eqref{eq:x2}.
\end{proof}


\subsection{Proof of Theorem \ref{thm:bbl}} We can now give the proof of our first main result. As announced, we shall show that the conclusion of Theorem \ref{thm:bbl} holds for all hereditarily convex even measures.

\begin{theorem} \label{thm:bbl-gen}
Let $W:\R^n\to\R$ be a $C^2$ smooth even function and assume that the measure $\mu$ on $\R^n$ with density $e^{-W(x)}$ is hereditarily convex. Moreover, let $\Omega\subseteq \R^{n+1}$ be a convex set such that for each $t\in\R$, the set $\Omega_t=\{x\in\R^n: \ (t,x)\in\Omega\}$ is symmetric, and $\Phi:\Omega\to\R_+$ a concave function such that for each $t$, the function $\Phi(t,\cdot)$ is even on $\Omega_t$.  Then, 
\begin{equation} \label{eq:varphi-thm-gen}
\varphi(t) \eqdef \Bigg( \rmint_{\Omega_t} \Phi(t,x)^\beta\,\diff\mu(x)\Bigg)^{\frac{1}{\beta+n}}
\end{equation}
is concave on its support for every $\beta>0$, provided that the integral converges.
\end{theorem}

The proof of Theorem \ref{thm:bbl-gen} shall follow from the representation of the second derivative of $\varphi$ established in Proposition \ref{prop:nguyen} along with a series of approximation arguments (which were missing from \cite{Ngu14a} in the study of the unweighted analogue of Theorem \ref{thm:bbl}, where $\Omega$ was only assumed to be a product set). As it turns out, a key new idea here is to assume a priori that the domain $\Omega$ is a super-level-set of the function $\Phi$ as some of the problematic boundary terms in \eqref{eq:nguyen} then become meaningful.  In other words, we start by considering functions of the form
\begin{equation}
\varphi(t) \eqdef \Bigg(\rmint \big(\Phi(t,x)-\eta\big)_+^\beta \,\diff\mu(x)\Bigg)^{\frac{1}{\beta+n}}
\end{equation}
for rather smooth $\Phi$ and then we proceed by approximation. We present the proof in a few steps.

\medskip

\noindent {\bf Step 1.} Further to the assumptions of Theorem \ref{thm:bbl-gen}, assume that:
\begin{enumerate}
\item[A.] The convex set $\Omega$ is open, bounded and strictly convex with $C^3$ smooth boundary;
\item[B.]  The function $\Phi$ is $C^2$ smooth on $\overline{\Omega}$; 
\item[C.] There exists $\eta>0$ such that if $\Sigma = \partial\Omega$, then
\begin{equation} \label{eq:super-level}
\forall \ (t,x)\in\Sigma, \qquad \Phi(t,x) \equiv \eta.
\end{equation}
\end{enumerate}
Then, the function \eqref{eq:varphi-thm-gen} is concave on its support.

\medskip

As concavity is a local property, it suffices to assume that $\Omega_t$ is nonempty for $t\in(-\eta,\eta)$ and prove that $\varphi''(0)\leq 0$.  For the proof of Step 1, we shall use the following elementary lemma.

\begin{lemma} \label{lem:level-set}
Under the assumptions of Step 1,  if $\Sigma_0=\partial\Omega_0$,  then we have
\begin{equation} \label{eq:leve}
\forall \ y\in\Sigma_0,\qquad \partial_t \Phi(0,y) = \delta(y) \psi(y),
\end{equation}
where $\delta(y)\eqdef - \langle \nabla_x\Phi(0,y), \hat{\nu}_{\Sigma_0}(y)\rangle \geq 0$ and $\psi\equiv\psi_0:\Sigma_0\to\R$ is the variation appearing in \eqref{eq:Minkowski-var}. 
\end{lemma}

\begin{proof}
Suppose that $\gamma:(-\e,\e)\to\Sigma_t$ is a smooth curve with $\gamma(0)=x$. Then, differentiating the assumption \eqref{eq:super-level}, we get that
\begin{equation}
0= \frac{\diff}{\diff s}\bigg|_{s=0} \Phi\big(t, \gamma(s)\big) = \left\langle \nabla_x \Phi(t,x) ,  \gamma'(0) \right\rangle.
\end{equation}
Since this holds for an arbitrary tangent curve $\gamma$, we deduce that
\begin{equation} \label{eq:para}
\nabla_x\Phi(t,x) = \langle \nabla_x\Phi(t,x),\hat{\nu}_{\Sigma_t}(x) \rangle \cdot \hat{\nu}_{\Sigma_t}(x),
\end{equation}
where $\hat{\nu}_{\Sigma_t}(x)$ is the outward pointing unit normal of $\Sigma_t$ at $x$.  Moreover, as $\Phi(t,\cdot)$ is even and concave on $\overline{\Omega_t}$, we conclude that the function $[0,1]\ni r \mapsto \Phi(t,rx)$ is decreasing. Therefore,
\begin{equation}
0\geq \frac{\diff}{\diff r} \bigg|_{r=1} \Phi(t,rx) = \left\langle \nabla_x\Phi(t,x), x\right\rangle \stackrel{\eqref{eq:para}}{=} \langle \nabla_x\Phi(t,x),\hat{\nu}_{\Sigma_t}(x) \rangle \left\langle x, \hat{\nu}_{\Sigma_t}(x)\right\rangle,
\end{equation}
which readily implies $\langle \nabla_x\Phi(t,x),\hat{\nu}_{\Sigma_t}(x) \rangle\leq0$ by the convexity of $\overline{\Omega_t}$.

If $y\in\Sigma_0$ and $x=F(t,y)\in\Sigma_t$, the assumption \eqref{eq:super-level} also gives
\begin{equation} \label{eq:para2}
\begin{split}
0 & =  \frac{\partial}{\partial t} \Phi\big(t, F(t,y)\big) = \partial_t \Phi(t,x) + \left\langle \nabla_x\Phi(t,x), \frac{\partial}{\partial t} F(t,y)\right\rangle
\\ & \stackrel{\eqref{eq:Minkowski-var}\wedge\eqref{eq:para}}{=} \partial_t \Phi(t,x) + \left\langle  \langle \nabla_x\Phi(t,x),\hat{\nu}_{\Sigma_t}(x) \rangle \cdot \hat{\nu}_{\Sigma_t}(x),  \psi_t(y) \hat\nu_{\Sigma_0}(y) + \mr{I\!I}_{\Sigma_0}^{-1}(y) \nabla_{\Sigma_0} \psi_t(y)  \right\rangle 
\\ & \stackrel{\eqref{eq:compat}}{=}  \partial_t \Phi(t,x) + \left\langle  \langle \nabla_x\Phi(t,x),\hat{\nu}_{\Sigma_t}(x) \rangle \cdot \hat{\nu}_{\Sigma_0}(y),  \psi_t(y) \hat\nu_{\Sigma_0}(y) + \mr{I\!I}_{\Sigma_0}^{-1}(y) \nabla_{\Sigma_0} \psi_t(y)  \right\rangle
\\ & = \partial_t \Phi(t,x) +  \langle \nabla_x\Phi(t,x),\hat{\nu}_{\Sigma_0}(y) \rangle \psi_t(y),
\end{split}
\end{equation}
which confirms the desired identity \eqref{eq:leve} for $t=0$.
\end{proof}

We are now ready to complete the proof of Step 1.

\begin{proof} [Proof of Step 1]
We shall assume that $\Omega_t$ is nonempty for $t\in(-\eta,\eta)$ and prove that $\varphi''(0)\leq 0$.  As before, we denote by $\Phi_0 =  \Phi(0,\cdot)$, $\partial_t\Phi_0 = \partial_t\Phi(0,\cdot)$,  $\psi\equiv\psi_0$ as in \eqref{eq:Minkowski-var},  $\nu$ as in \eqref{eq:u} and $u$ satisfying \eqref{eq:nguyen-eq} and \eqref{eq:nguyen-neumann}.  The assumptions of Proposition \ref{prop:nguyen} are satisfied, so we can use identity \eqref{eq:nguyen} for $\varphi''(0)$. 

By the joint concavity of $\Phi$, we have
\begin{equation}
\frac{\langle\nabla^2_{(t,x)}\Phi(0,\cdot) V,V\rangle}{\Phi(0,\cdot)} \leq 0
\end{equation}
for $V\in\R^{n+1}$, so the term in the first line of \eqref{eq:nguyen} is nonpositive. Moreover, fixing $y\in\Sigma_0$ and any $\msf{u},\msf{v}$ in the tangent space $\msf{T}_y\Sigma_0$, the Cauchy--Schwarz inequality for positive quadratic forms gives
\begin{equation}
\langle \mr{I\!I}_{\Sigma_0} \msf{u}, \msf{u}\rangle +\langle \mr{I\!I}_{\Sigma_0}^{-1}\msf{v},\msf{v}\rangle \geq- 2 \langle \msf{v}, \msf{u}\rangle.
\end{equation}
Therefore, the last term of \eqref{eq:nguyen} is also nonpositive choosing $\msf{u} = \beta\nabla_{\Sigma_0} u(y)$ and $\msf{v} = \nabla_{\Sigma_0} \psi(y)$. The hereditary convexity assumption for $\mu$ (in the dimensional form \eqref{eq:x2n}) combined with Jensen's inequality further give the estimate
\begin{equation} \label{eq:use-her}
\begin{split}
-\beta^2\rmint_{\Omega_0} \|\nabla^2u\|_{\mr{HS}}^2 & +  \langle \nabla^2W\nabla u, \nabla u \rangle\,\diff\nu  - \beta \rmint_{\Omega_0} \big(\ms{L}_\mu u\big)^2\,\diff\nu
 \stackrel{\eqref{eq:x2n}}{\leq} - \frac{\beta(\beta+n)}{n} \Bigg(  \rmint_{\Omega_0} \ms{L}_\mu u\,\diff\nu\Bigg)^2.
\end{split}
\end{equation}
Adopting the notations
\begin{equation}
A= \rmint_{\Omega_0} \ms{L}_\mu u\,\diff\nu, \quad B=\rmint_{\Omega_0} \frac{\partial_t \Phi_0}{\Phi_0}\,\diff\nu, \quad \mbox{and} \quad C=\rmint_{\Sigma_0} \psi \,\diff\nu,
\end{equation}
omitting the above nonpositive terms and using \eqref{eq:use-her}, we get from \eqref{eq:nguyen} with $\gamma=\frac{1}{\beta+n}$ that
\begin{alignat*}{2}
\frac{1}{\gamma}\frac{\varphi''(0)}{\varphi(0)}& \leq - \frac{\beta(\beta+n)}{n} A^2 - \frac{\beta n}{\beta+n} B^2  - \frac{n}{\beta(\beta+n)} C^2 -2\beta AB&& -2AC - \frac{2n}{\beta+n} BC 
\\ & &&\!\!\!\! -  \beta\!\rmint_{\Sigma_0} \frac{\partial_t\Phi_0}{\Phi_0}\cdot\psi\diff\nu + \rmint_{\Sigma_0} \frac{\partial}{\partial t}\bigg|_{t=0}\psi_t \,\diff\nu
\\ & = -\left(\sqrt{\frac{\beta(\beta+n)}{n}} A+\sqrt{\frac{\beta n}{\beta+n}}B+\sqrt{\frac{n}{\beta(\beta+n)}}C \right)^2 
&&\!\!\!\! -  \beta\!\rmint_{\Sigma_0}\frac{\partial_t\Phi_0}{\Phi_0}\cdot\psi\,\diff\nu + \rmint_{\Sigma_0} \frac{\partial}{\partial t}\bigg|_{t=0}\psi_t \,\diff\nu
\\ & \leq - \beta\rmint_{\Sigma_0}\frac{\partial_t\Phi_0}{\Phi_0}\cdot\psi\,\diff\nu + \rmint_{\Sigma_0} \frac{\partial}{\partial t}\bigg|_{t=0}\psi_t \,\diff\nu.
\end{alignat*}
To bound this remaining terms,  we resort to Lemma \ref{lem:level-set}.  Indeed, by \eqref{eq:super-level} and \eqref{eq:leve},
\begin{equation} \label{eq:first-pro}
- \beta\rmint_{\Sigma_0}\frac{\partial_t\Phi_0}{\Phi_0}\cdot\psi\,\diff\nu =  -\frac{\beta}{\eta} \rmint_{\Sigma_0} \delta(y) \cdot \psi(y)^2\,\diff\nu(y) \leq0
\end{equation}
and thus
\begin{equation} \label{eq:one-left}
\frac{1}{\gamma}\frac{\varphi''(0)}{\varphi(0)} \leq \rmint_{\Sigma_0} \frac{\partial}{\partial t}\bigg|_{t=0}\psi_t \,\diff\nu.
\end{equation}
Finally, to bound this last term observe that in the terminology of Lemma \ref{lem:Minkowski-var},  for $y\in\Sigma_0$,
\begin{equation} \label{eq:part-psi}
\begin{split}
\frac{\partial}{\partial t}\bigg|_{t=0} \psi_t(y) = \frac{\partial^2}{\partial t^2}\bigg|_{t=0} g(t,y) & \stackrel{\eqref{eq:defn-g}}{=}  \frac{\partial^2}{\partial t^2}\bigg|_{t=0} h_{\Omega_t}\big(\hat{\nu}_{\Sigma_0}(y)\big) =  \frac{\partial^2}{\partial t^2}\bigg|_{t=0} \left\langle \hat{\nu}_{\Sigma_t}^{-1} \circ \hat{\nu}_{\Sigma_0}(y) , \hat{\nu}_{\Sigma_0}(y) \right\rangle
\\ & \stackrel{\eqref{eq:defn-F}}{=}  \frac{\partial^2}{\partial t^2}\bigg|_{t=0} \left\langle F(t,y), \hat{\nu}_{\Sigma_0}(y)\right\rangle = \left\langle \frac{\partial^2}{\partial t^2}\bigg|_{t=0} F(t,y), \hat{\nu}_{\Sigma_0}(y)\right\rangle.
\end{split}
\end{equation}
Differentiating the first line of \eqref{eq:para2}, we get
\begin{alignat*}{2}
0 & = \frac{\partial}{\partial t}\bigg|_{t=0} &&\!\!\!\!\!\!\! \bigg\{ \partial_t\Phi\big(t, F(t,y)\big) + \left\langle \nabla_x\Phi\big(t,F(t,y)\big), \frac{\partial}{\partial t} F(t,y) \right\rangle\bigg\} 
\\ & = \partial_{tt}\Phi(0,y) && + 2 \left\langle \nabla_x\partial_t\Phi(0,y), \frac{\partial}{\partial t}\bigg|_{t=0} F(t,y) \right\rangle 
\\ & &&+ \left\langle \nabla_x^2\Phi(0,y) \frac{\partial}{\partial t}\bigg|_{t=0} F(t,y), \frac{\partial}{\partial t}\bigg|_{t=0} F(t,y)\right\rangle + \left\langle \nabla_x\Phi(0,y), \frac{\partial^2}{\partial t^2}\bigg|_{t=0} F(t,y) \right\rangle
\\ & \stackrel{\eqref{eq:para}}{=} &&\!\!\!\!\!\!\!\!\!\!\!\!\!\!\!\!\!\!\!\!\!\!  \left\langle \nabla_{(t,x)}^2  \Phi(0,y) \bigg(1,\frac{\partial}{\partial t}\bigg|_{t=0} F(t,y)\bigg), \bigg(1,\frac{\partial}{\partial t}\bigg|_{t=0} F(t,y)\bigg)\right\rangle
\\ & &&  + \langle \nabla_x\Phi(0,y),\hat{\nu}_{\Sigma_0}(y)\rangle  \left\langle \frac{\partial^2}{\partial t^2}\bigg|_{t=0} F(t,y), \hat{\nu}_{\Sigma_0}(y)\right\rangle
\\ & \leq - \delta(y) && \!\!\!\!\!\!\!\!\! \left\langle  \frac{\partial^2}{\partial t^2}\bigg|_{t=0} F(t,y), \hat{\nu}_{\Sigma_0}(y)\right\rangle,
\end{alignat*}
where the last inequality follows by the concavity of $\Phi$ and the definition of $\delta(y)$ from Lemma \ref{lem:level-set}. Combining this with \eqref{eq:one-left} and \eqref{eq:part-psi}, we conclude that $\varphi''(0)\leq0$ by the positivity of $\delta(y)$.
\end{proof} 

\medskip

\noindent {\bf Step 2.} Further to the assumptions of Theorem \ref{thm:bbl-gen}, assume that:
\begin{enumerate}
\item[A.] The convex set $\Omega$ is open and bounded;
\item[B.] The function $\Phi$ is continuous on $\overline{\Omega}$ and there exists $\eta>0$ such that if $\Sigma= \partial\Omega$, then
\begin{equation} \label{eq:super-level-2}
\forall \ (t,x)\in\Sigma, \qquad \Phi(t,x) \equiv \eta.
\end{equation}
\end{enumerate}
Then, the function \eqref{eq:varphi-thm-gen} is concave on its support.

\medskip

\begin{proof} [Proof of Step 2.]

Notice that,  unlike Step 1, we do not assume any regularity for the function $\Phi$ or the boundary of $\Omega$.  For $\e>0$ small enough, consider the function $\Phi_\e:\overline{\Omega}\to\R_+$ given by
\begin{equation}
\forall \ (t,x)\in\overline{\Omega}, \qquad \Phi_\e(t,x) = \Phi\big(t,(1-\e)x\big)
\end{equation}
and observe that it satisfies
\begin{equation}
\max_{(t,x)\in\overline{\Omega}}\big|\Phi_\e(t,x)-\Phi(t,x)\big| = \max_{(t,x)\in\overline{\Omega}}\big|\Phi\big(t,(1-\e)x\big)-\Phi(t,x)\big| \leq \omega_\Phi(\e\cdot \mr{diam}\Omega),
\end{equation}
where $\omega_\Phi$ is the modulus of continuity of $\Phi$ which satisfies $\lim_{\delta\to0^+}\omega_\Phi(\delta)=0$ as $\Phi$ is continuous (and thus uniformly continuous) on the compact set $\overline{\Omega}$.

Moreover,  let $\eta_\e:\msf{B}^{n+1}_2(0,\e)\to\R_+$ be a $C^\infty$ smooth mollifier which is even in the $x$-coordinates and is supported on the ball of radius $\e$. Then, consider the function
\begin{equation}
\begin{split}
\forall \ (t,x)\in\overline{\Omega},\qquad \Psi_\e(t,x) \eqdef & \big(\Phi_\e\ast\eta_e\big)(t,x) -\e(t^2+|x|^2)
\\ & = \rmint_{\msf{B}_2^{n+1}(0,\e)} \eta_\e(s,y) \Phi_\e(t-s,x-y)\,\diff s\diff y -\e(t^2+|x|^2).
\end{split}
\end{equation}
It is immediate that $\Psi_\e$ is $C^\infty$ smooth on $\overline{\Omega}$,  strictly concave and even in the $x$-coordinates.  Moreover,
\begin{equation*}
\begin{split}
\max_{(t,x)\in\overline{\Omega}}& \big| \Psi_\e(t,x)  - \Phi_\e(t,x)\big|
\\ & \leq \max_{(t,x)\in\overline{\Omega}} \rmint_{\msf{B}_2^{n+1}(0,\e)} \eta_\e(s,y)  \big| \Phi_\e(t-s,x-y) - \Phi_\e(t,x)\big| \,\diff s\diff y + \e\mr{diam}\Omega^2
 \leq \omega_{\Phi}(\e) + \e\mr{diam}\Omega^2.
\end{split}
\end{equation*}
Putting everything together, we conclude that the sequence of $C^\infty$ smooth strictly concave functions $\{\Psi_{1/k}:\overline{\Omega}\to\R_+\}_{k\geq1}$ converges uniformly to $\Phi$ on $\overline{\Omega}$.  

What remains is to ensure that $\Psi_{1/k}$ satisfies the boundary condition \eqref{eq:super-level} when restricted to an appropriate set. To that end, consider the open sets
\begin{equation}
\Omega_{k,m} = \big\{(t,x) \in{\Omega}: \ \Psi_{1/k}(t,x)> \eta+\tfrac{1}{m}\big\},
\end{equation}
which are additionally strictly convex by the strict concavity of $\Psi_{1/k}$.  Unless $\Phi$ is the constant function equal to $\eta$ (in which case the conclusion follows easily from Step 1 by approximating the domain $\Omega$ by smooth domains),  the sets $\Omega_{k,m}$ satisfy
\begin{equation} \label{eq:appr-dom}
\bigcup_{k,m\geq1} \Omega_{k,m}  = \Omega
\end{equation}
and thus they are nonempty open strictly convex sets for large enough $k,m\in\N$. Moreover,  $\Psi_{1/k}(t,x)\equiv \eta+\tfrac{1}{m}$ on the boundary of $\Omega_{k,m}$ which additionally implies that this boundary is $C^\infty$ smooth as $\nabla\Psi_{1/k}\neq 0$ on $\partial \Omega_{k,m}$.  To see this, observe that if indeed $\nabla\Psi_{1/k}(x,t)=0$ for some $(x,t)\in\partial\Omega_{k,m}$, then the condition $\Psi_{1/k}(x,y) = \eta+\tfrac1m$ and the strict concavity of $\Psi_{1/k}$ imply that $\Psi_{1/k} < \eta+\tfrac1m$ in a small  ball around $(x,t)$ which contradicts that $(x,t) \in \partial \Omega_{k,m}$

Therefore, the assumptions of Step 1 are satisfied for the functions $\Psi_{1/k}$ restricted to $\overline{\Omega}_{k,m}$ when $k,m$ are large enough. As the conclusion of Theorem \ref{thm:bbl-gen} is stable under uniform limits and \eqref{eq:appr-dom} shows that the corresponding domains coincide in the limit,  this concludes the proof of Step 2.
\end{proof}

\medskip

\noindent {\bf Step 3.}  Further to the assumptions of Theorem \ref{thm:bbl-gen}, assume that:
\begin{enumerate}
\item[A.] The convex set $\Omega$ is open and bounded;
\end{enumerate}
Then, the function \eqref{eq:varphi-thm-gen} is concave on its support.

\medskip

To complete the proof of Step 3, we shall use a simple extension lemma for concave functions.

\begin{lemma} \label{lem:extension}
Let $\Omega, \Theta$ be bounded open convex sets in $\R^{n+1}$ such that $\overline{\Theta} \subseteq \Omega$ and consider a concave function $\Phi:\overline{\Omega}\to [0,\infty)$. Then, there exists a concave function $\widetilde{\Phi}:\overline{\Omega}\to[0,\infty)$ such that $\widetilde{\Phi}\equiv \Phi$ in $\overline{\Theta}$ and $\widetilde{\Phi}\equiv 0$ on $\partial\Omega$.
\end{lemma}

\begin{proof}
Consider the function $\widetilde{\Phi}:\overline{\Omega}\to[0,\infty)$ given by
\begin{equation*}
\forall \ x\in\overline{\Omega},\qquad \widetilde{\Phi}(x) = \sup\big\{ \lambda \Phi(z): \ \lambda\in[0,1], \ z\in \overline{\Theta} \ \mbox{for which } \exists \ y\in \overline{\Omega} \mbox{ with } x=\lambda z+(1-\lambda)y\big\}.
\end{equation*}
Observe that the set in the definition of $\widetilde{\Phi}$ is nonempty as $\Omega$ is bounded.  It is clear that $\widetilde{\Phi}(x)\geq \Phi(x)$ for every $x\in\overline{\Theta}$ and conversely if $x\in\overline{\Theta}$ can be written as $x=\lambda z+(1-\lambda)y$ for some $\lambda\in[0,1]$, $z\in \overline{\Theta}$ and $y\in\overline{\Omega}$, then
\begin{equation}
\lambda\Phi(z) \leq \lambda \Phi(z)+(1-\lambda)\Phi(y) \leq \Phi\big(\lambda z+(1-\lambda) y\big) = \Phi(x)
\end{equation}
by the nonnegativity and the concavity of $\Phi$. This shows that $\widetilde{\Phi}\equiv\Phi$ on $\overline{\Theta}$. Moreover, if $x\in\partial\Omega$ is such that $x=\lambda z+(1-\lambda)y$ for some $\lambda\in[0,1]$, $z\in \overline{\Theta}$ and $y\in\overline{\Omega}$, then we conclude that $\lambda=0$ as otherwise we would have $x\in\Omega$ due to the inclusion $\overline{\Theta}\subseteq \Omega$. Thus, $\widetilde{\Phi}\equiv0$ on $\partial\Omega$.

Finally, to verify that $\widetilde{\Phi}$ is concave, consider $x_1,x_2\in\overline{\Omega}$ and for $\e>0$ fix $\lambda_i\in[0,1]$, $z_i\in\overline{\Theta}$ and $y_i\in\overline{\Omega}$ such that $x_i = \lambda_i z_i + (1-\lambda_i) y_i$ and
\begin{equation}
\lambda_i \Phi(z_i) \geq \widetilde{\Phi}(x_i) - \e
\end{equation}
for $i\in\{1,2\}$. Then, for $\mu\in[0,1]$,  we can write $\mu x_1+(1-\mu)x_2$ as
\begin{equation*}
\big(\mu\lambda_1+(1-\mu)\lambda_2\big) \frac{\mu\lambda_1 z_1+(1-\mu)\lambda_2 z_2}{\mu\lambda_1+(1-\mu)\lambda_2} + \big(\mu(1-\lambda_1)+(1-\mu)(1-\lambda_2)\big) \frac{\mu(1-\lambda_1)y_1+(1-\mu)(1-\lambda_2)y_2}{\mu(1-\lambda_1)+(1-\mu)(1-\lambda_2)}
\end{equation*}
and using the definition of $\widetilde{\Phi}$ and the concavity of $\Phi$, we get
\begin{equation}
\begin{split}
\widetilde{\Phi}\big(\mu x_1+(1-\mu)x_2\big) \geq \big(\mu\lambda_1+&(1-\mu)\lambda_2\big)  \Phi\bigg( \frac{\mu\lambda_1 z_1+(1-\mu)\lambda_2 z_2}{\mu\lambda_1+(1-\mu)\lambda_2}\bigg)
\\ & \geq \mu\lambda_1\Phi(z_1) + (1-\mu)\lambda_2\Phi(z_2) \geq \mu \widetilde{\Phi}(x_1) + (1-\mu)\widetilde{\Phi}(x_2) - \e.
\end{split}
\end{equation}
Letting $\e\to0^+$, we conclude that $\widetilde{\Phi}$ is indeed concave on $\overline{\Omega}$.
\end{proof}

\begin{proof}[Proof of Step 3]
For $\e>0$, let $\Omega_{-\e}$ be an open convex set satisfying $\overline{\Omega}_{-\e} \subseteq \Omega \subseteq \overline{\Omega}\subseteq  \Omega_{-\e}+\e\msf{B}_2^{n+1}$ and such that each section $\Omega_{-\e,t}\eqdef \{x: \ (t,x) \in\Omega_{-\e}\}$ is origin symmetric.  Using Lemma \ref{lem:extension}, consider a concave function $\Phi_\e:\overline{\Omega}\to[0,\infty)$ such that $\Phi_\e \equiv \Phi$ on $\overline{\Omega}_{-\e}$ and $\Phi_\e \equiv 0$ on $\partial\Omega$.  Then, by Step 2,  for $\beta>0$, the function 
\begin{equation}
\varphi_\e(t) \eqdef \Bigg( \rmint_{\Omega_{t}}\big( \Phi_\e(t,x)+\e\big)^\beta\,\diff\mu(x)\Bigg)^{\frac{1}{\beta+n}}
\end{equation}
is concave on its support.  Letting $\e\to0^+$,  $\varphi_\e(t) \to\varphi(t)$ pointwise so the result follows.
\end{proof}

Finally, we are ready to conclude the proof of Theorem \ref{thm:bbl-gen}.

\begin{proof} [Proof of Theorem \ref{thm:bbl-gen}]
By Step 3, the result is known if $\Omega$ is open and bounded. The conclusion follows by approximating $\Omega$ by the domains $\Omega_m = \mathrm{int}\big(\Omega\cap \msf{B}_2^{n+1}(0,m)\big)$ and letting $m\to\infty$.
\end{proof}

\begin{proof} [Proof of Theorem \ref{thm:bbl}]
This is an immediate consequence of Theorem \ref{thm:bbl-gen} and Proposition \ref{prop:x2} after approximating $\mu$ by a smooth enough measure of full support on $\R^n$.
\end{proof}

As we already mentionned,  Theorem \ref{thm:bbl} also implies new weighted Brunn--Minkowski inequalities.

\begin{proof} [Proof of Corollary \ref{cor:b-conc}]
This is an application of Theorem \ref{thm:bbl} for the convex set
\begin{equation}
\Omega\eqdef\bigcup_{t\in[0,1]} \{t\}\times \big(tK+(1-t)L\big) \subseteq \R^{n+1}
\end{equation}
and the concave function $\Phi(t,x)\eqdef \Phi(x){\bf 1}_C(x)$, where $(t,x)\in\Omega$.
\end{proof}


\section{Log-concavity of weighted marginals and the $B$-inequality} \label{sec:log-conc}

The case of log-concavity corresponds formally to the case $\beta=\infty$. 

\begin{proof} [Proof of Theorem \ref{cor:prekopa}]
As before, we can assume that $\mu$ has a smooth density $e^{-W}$ of full support. We want to compute the second derivative of the function
\begin{equation}
\varphi(t)=\log \bigg( \rmint_{\R^n} e^{-V(t,\cdot)}\, \diff\mu\bigg) = \log \bigg( \rmint_{\R^n} e^{-V(t,x)-W(x)}\, \diff\mu(x)\bigg)
\end{equation}
and show that $\varphi''(t)\leq 0$. To that end, we introduce the probability measure
\begin{equation}
\diff\nu_t (x) = e^{-V(t,x)} \, \frac{\diff\mu(x)}{e^{\varphi(t)}},
\end{equation}
and a smooth function $u\equiv u_t:\R^n \to \R$
 such that
\begin{equation}
\ms{L}_{\nu_t} u = \partial_t V(t, \cdot)- \rmint_{\R^n} \partial_t V(t, \cdot)\, \diff\nu_t.
\end{equation}
The resulting formula for $\varphi''(t)$ is straightforward by manipulations similar to those of Proposition~\ref{prop:nguyen}, but we can also apply \eqref{eq:nguyen} with $\Phi(t,x)= (1-\frac1\beta V(t,x))_+$ and let $\beta\to \infty$,  $\beta\gamma\to 1$. This yields
\begin{equation}
-\frac{\varphi''(t)}{\varphi(t)}=
\rmint_{\R^n} \langle\nabla^2_{(t,x)} V X, X \rangle  \, \diff \nu_t 
+ \rmint_{\R^n} \|\nabla^2 u \|^2_{\mr{HS}} + 
\langle\nabla^2 W \nabla u, \nabla u\rangle \, \diff\nu_t,
\end{equation}
where $X(x)=(1,-\nabla u(x))$ for $x\in\R^{n+1}$. Using assumption \eqref{eq:assum-pr-cor}, we get
\begin{equation}
\begin{split}
\rmint_{\R^n}  \langle\nabla^2_{(t,x)} V X, X \rangle \,\diff\nu_t
\ge \kappa\bigg( \rmint_{\R^n} \langle\nabla_x V(t,x) , x  \rangle \,\diff&\nu_t(x)  -2 \rmint_{\R^n} \langle \nabla_x V(t, x) , \nabla u(x) \rangle\,\diff\nu_t(x) \bigg)
\\ & \ge -\kappa \frac{\Big(\rmint_{\R^n} \langle\nabla_x V(t, x) , \nabla u(x)\rangle\, \diff\nu_t(x)\Big)^2}{\rmint_{\R^n} \langle\nabla_x V(t,x) , x \rangle \, \diff\nu_t(x)},
\end{split}
\end{equation}
where we used that $V(t, \cdot)$ is convex and even to ensure that $\langle\nabla_x V(t,x) , x \rangle \ge 0$. Therefore,
\begin{equation} \label{eq:lcdone}
\begin{split}
-\frac{\varphi''(t)}{\varphi(t)} & \ge
\rmint_{\R^n} \|\nabla^2 u \|^2_{\mr{HS}} + 
\langle\nabla^2 W \nabla u, \nabla u\rangle \, \diff\nu_t
- \kappa  \frac{\Big(\rmint_{\R^n} \langle\nabla_x V(t, x) , \nabla u(x)\rangle\, \diff\nu_t(x)\Big)^2}{\rmint_{\R^n} \langle\nabla_x V(t,x) , x \rangle \, \diff\nu_t(x)}
\\ & \stackrel{\eqref{eq:x2}}{\geq} \frac{\Big( \rmint_{\R^n} \ms{L}_\mu u\,\diff\nu\Big)^2 }
{\rmint_{\R^n} \ms{L}_\mu (|x|^2/2) \,\diff\nu} - \kappa  \frac{\Big(\rmint_{\R^n} \langle\nabla_x V(t, x) , \nabla u(x)\rangle\, \diff\nu_t(x)\Big)^2}{\rmint_{\R^n} \langle\nabla_x V(t,x) , x \rangle \, \diff\nu_t(x)},
\end{split}
\end{equation}
where the second inequality follows because $\mu$ is hereditarily convex by Proposition \ref{prop:x2}. Recall however that $\ms{L}_{\nu_t} v = \ms{L}_\mu v - \langle \nabla V(t,\cdot), \nabla v\rangle$ and $\rmint_{\R^n} \ms{L}_{\nu_t}v\,\diff\nu_t=0$ for any smooth function $v$, which morally corresponds to a zero Neumann boundary condition at infinity.  Therefore,  \eqref{eq:lcdone} can be equivalently rewritten as
\begin{equation}
-\frac{\varphi''(t)}{\varphi(t)} \ge (1-\kappa) \frac{\Big(\rmint_{\R^n} \langle\nabla_x V(t, x) , \nabla u(x)\rangle\, \diff\nu_t(x)\Big)^2}{\rmint_{\R^n} \langle\nabla_x V(t,x) , x \rangle \, \diff\nu_t(x)} \geq 0,
\end{equation}
which concludes the proof of the theorem as $\kappa\leq 1$.
\end{proof}

\begin{remark} \label{rem:B-proofs}
In view of \eqref{eq:b-hessian}, the case $\kappa=1$ of Theorem \ref{cor:prekopa} contains the B-theorem for rotationally invariant measures.  Even though our proof, like that of \cite{CFM04,CR23}, relies on Poincar\'e-type spectral inequalities, it is in fact quite different from the existing arguments in the literature. To see this, let $\diff\mu(x)=e^{-W(x)}\,\diff x$ be a rotationally invariant measure and $V:\R^n\to\R$ an arbitrary convex function.  These works prove the functional version of the B-inequality, asserting that the function
\begin{equation} \label{eq:def-a}
\alpha(t) \eqdef \rmint_{\R^n} e^{-V(e^tx)-W(x)}\,\diff x
\end{equation}
is log-concave on $\R$, by applying the change of variables $y=e^t x$ and writing
\begin{equation} \label{eq:ch-var}
\alpha(t) = e^{nt} \rmint_{\R^n} e^{-V(y)-W(e^{-t}y)}\,\diff y.
\end{equation}
As the first factor is log-affine in $t$, one then computes the second derivative of the logarithm of the second factor and its concavity is equivalent to a refined Poincar\'e inequality for the even function $x\mapsto\langle \nabla W(x), x\rangle$, defined on $\R^n$.  In the proof presented here, we do not implement the change of variables \eqref{eq:ch-var} and instead work directly with the definition \eqref{eq:def-a} of $\alpha$, at the expense of having to apply the more refined functional inequality \eqref{eq:x2} to prove the desired log-concavity.
\end{remark}

\begin{remark} \label{rem:B-BL}
The Gaussian B-inequality assets that for every symmetric convex set $K$ in $\R^n$, the map $t\mapsto \gamma_n(e^tK)$ is log-concave. The validity of this inequality was first proposed by Banaszczyk (unpublished) and further popularized in Lata{\l}a's ICM survey \cite{Lat02}.  Afterwards, a proof which followed the scheme described in Remark \ref{rem:B-proofs} was found by Cordero-Erausquin, Fradelizi and Maurey \cite{CFM04}.  In particular, \cite{CFM04} used the $L_2$ method to derive a refinement of the log-concave
Poincar\'e inequality for even functions that follows from~\eqref{eq:bl-var} or from $\Gamma_2$ principles (see \cite[Theorem~4.8.4]{BGL14}).

It was thus to our great surprise to discover the strong form of the Gaussian B-inequality fully proven in the conference proceedings publication \cite[Corollary~1.7]{BL75} of Brascamp and Lieb, which served as an announcement to the famous paper \cite{BL76}! Even more so, their proof is an extremely elegant argument which seems to have eluded the community for decades.   Despite the fact that the article \cite{BL75} has become accessible as part of the edited volume \cite{Lie02}, we shall repeat Brascamp and Lieb's proof of the Gaussian B-inequality in the Appendix for the readers' convenience.

That being said, although~\cite{CFM04} can no longer be credited for the (first!) solution to the B-conjecture for the Gaussian measure, it keeps the merit of having put forward $L_2$ methods in functional Brunn-Minkowski theory, and having related inequalities under symmetry with \emph{second} eigenvalue spectral problems, two ideas that apply beyond the Gaussian case, and that have proved to be central in many recent developments, including the present one. 
\end{remark}


\section{Proof of Theorem \ref{thm:poincare}} \label{sec:poincare}

We will follow a simple approximation argument due to \cite{Ngu14b} to\mbox{ derive Theorem \ref{thm:poincare} from Theorem \ref{thm:bbl}.}

\begin{proof} [Proof of Theorem \ref{thm:poincare}]
By suitable approximation, we shall assume that $C$ is a bounded open convex set and that both $f$ and $\Phi$ are $C^2$ smooth and symmetric in a neighborhood of $\overline{C}$, with $\Phi$ and $-\nabla^2\Phi$ moreover being strictly positive. If $\e>0$, consider the function $\Phi_\e:\R\times C\to\R$ given by
\begin{equation}
\forall \ (t,x)\in\R\times C,\qquad \Phi_\e(t,x) \eqdef \Phi(x) + t g(x) - \frac{t^2}{2} \big\langle \big( -\nabla^2 \Phi(x)\big)^{-1} \nabla g(x), \nabla g(x)\big\rangle - \frac{\e}{2}\big( |x|^2+t^2\big).
\end{equation}
It is clear from its definition and the strict positivity of $\Phi$ and $-\nabla^2\Phi$ that for small enough $\e>0$ and $\eta>0$, the restriction $\Phi_\e|_{(-\eta,\eta)\times C}$ is also strictly positive. Moreover, the Hessian of $\Phi_\e$ satisfies
\begin{equation}
\nabla^2\Phi_\e(0,x) = \begin{pmatrix} - \langle ( -\nabla^2 \Phi(x))^{-1} \nabla g(x), \nabla g(x)\rangle  & \nabla g(x)^* \\ \nabla g(x) & \nabla^2\Phi(x)  \end{pmatrix} - \e \msf{Id}.
\end{equation}
Therefore for $(u_0,\bar{u})\in\R^{n+1}$,
\begin{equation*}
\begin{split}
\big\langle \nabla^2\Phi_\e&(0,x) (u_0,\bar{u}), (u_0,\bar{u})\big\rangle
\\ & = -\big\langle \big( -\nabla^2 \Phi(x)\big)^{-1} u_0 \nabla g(x), u_0 \nabla g(x)\big\rangle + 2 \langle u_0\nabla g(x), \bar{u}\rangle - \big\langle \big(-\nabla^2\Phi(x) \big) \bar{u},\bar{u}\big\rangle - \e\big(u_0^2+|\bar{u}|^2\big) 
\\ & \leq - \e\big(u_0^2+|\bar{u}|^2\big)
\end{split}
\end{equation*}
by the Cauchy--Schwarz inequality $\langle Ma,a\rangle + \langle M^{-1}b,b\rangle \geq 2 \langle a,b\rangle$ which holds for an arbitrary positive definite matrix $M$ and vectors $a,b$. Thus, again by continuity, for sufficiently small $\e>0$ and $\eta>0$, $\Phi_\e|_{(-\eta,\eta)\times C}$ is also concave. Since for each $t$, $\Phi_\e(t,\cdot)$ is clearly even, we can invoke Theorem \ref{thm:bbl} to conclude that the function $\varphi_\e:(-\eta,\eta)\to \R_+$ given by
\begin{equation}
 \varphi_\e(t) \eqdef \Bigg( \rmint_{C} \Phi_\e(t,x)^\beta\,\diff\mu(x)\Bigg)^{\frac{1}{\beta+n}}
\end{equation}
is concave on $(-\eta, \eta)$. In particular, using \eqref{eq:second-derivative}, we derive
\begin{equation*}
\frac{\beta+n}{\beta} \frac{\varphi_\e''(0)}{\varphi_\e(0)} = \rmint_C \frac{\partial_{tt}\Phi_\e(0,\cdot)}{\Phi_\e(0,\cdot)}\,\diff\sigma_\e + (\beta-1)\mathrm{Var}_{\sigma_\e} \Bigg( \frac{\partial_t\Phi_\e(0,\cdot)}{\Phi_\e(0,\cdot)} \Bigg) -\frac{n}{\beta+n} \Bigg(\rmint_C\frac{\partial_t \Phi_\e(0,\cdot)}{\Phi_\e(0,\cdot)} \,\diff\sigma_\e\Bigg)^2 \leq 0,
\end{equation*}
where $\diff\sigma_\e(x)$ is the probability measure with density proportional to $\Phi_\e(0,x)^\beta = \big(\Phi(x)-\frac{\e}{2}|x|^2 \big)^\beta$ with respect to $\diff\mu(x)$. Plugging the formula for $\Phi_\e(0,x)$, along with $\partial_t \Phi_\e(0,x) = g(x)$ and
\begin{equation}
\partial_{tt}\Phi_\e(0,x) = -\big\langle \big( -\nabla^2 \Phi(x)\big)^{-1} \nabla g(x), \nabla g(x)\big\rangle - \e
\end{equation}
and letting $\e\searrow0$ (so that also $\sigma_\e \to\nu_\beta$ weakly), we recover
\begin{equation}
(\beta-1)\mathrm{Var}_{\nu_\beta}\Big( \frac{g}{\Phi}\Big) \leq \rmint_C \frac{\langle (-\nabla^2\Phi)^{-1} \nabla g, \nabla g\rangle }{\Phi} \,\diff\nu_\beta + \frac{n}{\beta+n} \Bigg( \rmint_C \frac{g}{\Phi} \,\diff\nu_\beta\Bigg)^2.
\end{equation}
This coincides with the desired inequality \eqref{eq:new-poin} as $g=f\Phi$.
\end{proof}


\section{Further remarks and open problems} \label{sec:last}


\subsection{Hereditary convexity}
The proof of Theorems \ref{thm:bbl} and \ref{cor:prekopa} relied on the functional inequality \eqref{eq:x2} that was proven for rotationally invariant measures in Proposition \ref{prop:x2}. It is clear from our proofs that the validity of such a functional inequality for general even log-concave measures would affirmatively resolve both the B-conjecture and the dimensional Brunn--Minkowski conjecture.  We therefore ask the following natural question.

\begin{question} \label{q:ce}
Let $\mu$ be an even log-concave measure supported on $\R^n$ with $C^2$ smooth density. Is $\mu$ hereditarily convex?
\end{question}


\subsection{From sets to functions} In \cite{CR20}, it was shown that the B-inequality
\begin{equation}
\forall \ a,b\in\R_+ \mbox{ and } \lambda\in(0,1),\qquad \mu\big(a^\lambda b^{1-\lambda}K\big) \geq\mu(aK)^\lambda \mu(bK)^{1-\lambda},
\end{equation}
where $K,L$ are symmetric convex sets, holds for all even log-concave measures on $\R^n$ for any $n\in\N$ if and only if the functional version of the inequality holds on any $\R^n$, namely if
\begin{equation} \label{eq:def-a-again}
\alpha(t) \eqdef \rmint_{\R^n} e^{-V(e^tx)-W(x)}\,\diff x
\end{equation}
is log-concave on $\R$ for all even convex functions $V,W:\R^n\to\R\cup\{\infty\}$. One could wonder whether a similar transference principle from geometric to functional inequalities also holds in the context of the dimensional Brunn--Minkowski inequality, thus showing that Theorem \ref{thm:bbl} is a formal consequence of the results of \cite{EM21,CR23}. While the multiplicative nature of the B-inequality appears to be what made the result of \cite{CR20} possible,  one can in fact recover concavity principles in the spirit of Theorem \ref{thm:bbl} from geometric inequalities for sets at least for \emph{integer} powers $\beta>0$.

\begin{theorem} \label{thm:transference}
Fix $n\in\N$ and an even log-concave measure $\mu$ on $\R^n$ such that for all $k\in\N$,
\begin{equation} \label{eq:assum-transference}
\forall \ \lambda\in(0,1),\qquad (\mu\otimes m^k)\big(\lambda K+(1-\lambda)L\big)^{\kappa_{n,k}} \geq \lambda(\mu\otimes m^k)(K)^{\kappa_{n,k}}+(1-\lambda)(\mu\otimes m^k)(L)^{\kappa_{n,k}}
\end{equation}
for some $\kappa_{n,k}>0$ and all symmetric convex sets $K, L$ in $\R^{n+k}$, where $m^k$ is the Lebesgue measure on $\R^k$. Moreover, let $\Omega\subseteq \R^{n+1}$ be a convex set such that for each $t\in\R$, the corresponding section $\Omega_t=\{x\in\R^n: \ (t,x)\in\Omega\}$ is symmetric, and $\Phi:\Omega\to\R_+$ a concave function such that for each $t$, the function $\Phi(t,\cdot)$ is even on $\Omega_t$.  Then, for every integer $\beta\in\N$, the function
\begin{equation}
 \varphi(t) \eqdef \Bigg( \rmint_{\Omega_t} \Phi(t,x)^\beta\,\diff\mu(x)\Bigg)^{\kappa_{n,\beta}}
\end{equation}
is concave on its support, provided that the integral converges.
\end{theorem}

\begin{proof}
Since $\beta\in\N$, there exists a positive constant $c_\beta>0$ such that
\begin{equation}
\Phi(t,x)^\beta = c_\beta m^\beta\big( \Phi(t,x) \msf{B}_2^\beta\big),
\end{equation}
where $\msf{B}_2^\beta$ is the unit Euclidean ball in $\R^\beta$. Therefore, we can write
\begin{equation}
\begin{split}
 \varphi(t) = c_\beta^{\kappa_{n,\beta}} (\mu \otimes m^\beta)\big( \underbrace{\big\{ (x,y)\in\R^{n+k}: \ x\in \Omega_t \mbox{ and } |y| \leq \Phi(t,x)\big\}}_{C(t)}\big)^{\kappa_{n,\beta}},
\end{split}
\end{equation}
where $|y|$ is the Euclidean norm of $y$. The conclusion now follows from \eqref{eq:assum-transference} by observing that, by the concavity of $\Phi$,  the sets $\{C(t)\}_{t\in[0,1]}$ are symmetric convex sets in $\R^{n+\beta}$ which satisfy
\begin{equation}
C\big(\lambda t_1 + (1-\lambda) t_2\big) \supseteq \lambda C(t_1) + (1-\lambda) C(t_2)
\end{equation}
for every $t_1,t_2 \in\R$ and $\lambda\in(0,1)$.
\end{proof}

Combining this result with the general bound of Livshyts \cite{Liv23} for the dimensional Brunn--Minkowski inequality for even log-concave measures, we deduce the following corollary.

\begin{corollary} \label{cor:livshyts}
Let $\mu$ be an even log-concave measure on $\R^n$. Moreover, let $\Omega\subseteq \R^{n+1}$ be a convex set such that for each $t\in\R$, the set $\Omega_t=\{x\in\R^n: \ (t,x)\in\Omega\}$ is symmetric, and $\Phi:\Omega\to\R_+$ a concave function such that for each $t$, the function $\Phi(t,\cdot)$ is even on $\Omega_t$.  Then, for every integer $\beta\in\N$, the function
\begin{equation}
 \varphi(t) \eqdef \Bigg( \rmint_{\Omega_t} \Phi(t,x)^\beta\,\diff\mu(x)\Bigg)^{\kappa_{n+\beta}}
\end{equation}
is concave on its support, where $\{\kappa_r\}_{r\in\N}$ is a sequence satisfying $\kappa_r \asymp r^{-4-o(1)}$ as $r\to\infty$.
\end{corollary}

An inspection of the proof of \cite{Liv23} shows that Livshyts in fact obtains the inequality
\begin{equation} 
\rmint_U \|\nabla^2 u\|_{\mr{HS}}^2 + \langle \nabla^2W \nabla u, \nabla u\rangle \, \diff\mu_t \geq \frac{1}{n^{4+o(1)}} \Bigg( \rmint_U \ms{L}_\mu u \,\diff\mu_t\Bigg)^2 
\end{equation}
for every even log-concave measure $\diff\mu(x)=e^{-W(x)}\,\diff x$ and every even $C^2$ function $u:U\to\R$. Therefore, a close inspection of the proof of Theorem \ref{thm:bbl} shows that the statement of Corollary \ref{cor:livshyts} can be extended from $\beta\in\N$ to arbitrary $\beta>0$.


\subsection{Weighted torsional rigidity} Inequalities of Brunn--Minkowski-type for functionals other than volume have been thoroughly studied throughout analysis. Prime examples from the works of Brascamp--Lieb and Borell include the Brunn--Minkowski inequality for the first eigenvalue of the Laplace operator \cite{BL76,Bor00}, for Newtonian \cite{Bor83} and logarithmic capacity \cite{Bor84} and for torsional rigidity \cite{Bor85}. For a unified perspective to these inequalities, we refer the reader to \cite{Col05}.

Let $K$ be a convex set in $\R^n$ with nonempty interior.  The \emph{torsional rigidity} $\tau(K)$ is defined by
\begin{equation} \label{eq:tor}
\frac{1}{\tau(K)} \eqdef\inf \Bigg\{ \frac{\rmint_K |\nabla u(x)|^2\,\diff x}{\big(\rmint_K |u(x)|\,\diff x\big)^2}: \ u \mbox{ is smooth and nonconstant in } K \mbox{ with } u\big|_{\partial K}\equiv 0 \Bigg\}.
\end{equation}
In \cite{Bor85}, Borell showed that for any convex sets $K,L$ in $\R^n$ with nonempty interior,
\begin{equation} \label{eq:borelll}
\forall \ \lambda\in(0,1),\qquad \tau\big( \lambda K+(1-\lambda)L\big)^{1/(n+2)} \geq \lambda \tau(K)^{1/(n+2)}+(1-\lambda)\tau(L)^{1/(n+2)}.
\end{equation}
In his argument, he made crucial use of an expression for $\tau(\cdot)$ in terms of the solution of an elliptic boundary-value problem. Namely, if $u:K\to\R$ is the unique solution of
\begin{equation} \label{eq:bvp}
\begin{cases}
\Delta u = -1, \quad \mbox{in } \mr{int}(K) \\
u \equiv 0, \quad \mbox{on } \partial K
\end{cases},
\end{equation}
then
\begin{equation} \label{eq:tor-for}
\tau(K) = \rmint_K |\nabla u(x)|^2\,\diff x \stackrel{\eqref{eq:bvp}}{=}  \rmint_K u(x)\,\diff x,
\end{equation}
where the second equality follows by integration by parts.
In \cite{Bor85}, Borell proved that if $K,L$ are convex sets in $\R^n$ with nonempty interior and $u(\lambda,\cdot):\lambda K+(1-\lambda)L \to\R$ is a solution to the boundary-value problem \eqref{eq:bvp} on $\lambda K+(1-\lambda)L$, then
the function $\sqrt{u(\lambda,x)}$ is concave in the pair of variables $(\lambda, x)$.  Afterwards, inequality \eqref{eq:borelll} becomes an immediate consequence of \eqref{eq:tor-for} and the concavity principle for the marginal \eqref{eq:varphi} with $\beta=2$.  We refer, for instance, to \cite[Section~4]{Sal15} for a comprehensive study of concavity properties of solutions to more general boundary-value problems.

In view of the recent interest in Brunn--Minkowski inequalities for measures of symmetric convex sets, it is worthwhile to investigate such inequalities for weighted analogues of other analytic functionals in the presence of symmetry.  Given an even probability measure $\mu$ on $\R^n$, the $\mu$-weighted torsional rigidity $\tau_\mu(K)$ of a symmetric convex set $K$ is defined by \eqref{eq:tor}, where all $\diff x$ integrals are replaced by $\diff\mu(x)$.  It can also be rephrased in terms of the solution to the boundary value problem
\begin{equation} \label{eq:bvp2}
\begin{cases}
\ms{L}_\mu u = -1, \quad \mbox{in } \mr{int}(K) \\
u \equiv 0, \quad \mbox{on } \partial K
\end{cases}.
\end{equation}
We refer to \cite{Liv24,HL24} for some recent investigations on the functional $\tau_\mu$. In view of the weighted concavity principle of Theorem \ref{thm:bbl}, an affirmative answer to the following question would prove the weighted analogue of \eqref{eq:borelll} for rotationally invariant log-concave weights.

\begin{question}
Fix an even log-concave measure $\mu$ on $\R^n$ and let $K, L$ be symmetric convex sets in $\R^n$.  Let also $u(\lambda,\cdot):\lambda K+(1-\lambda)L \to\R$ be the unique solution to \eqref{eq:bvp2} on $\lambda K+(1-\lambda)L$, where $\lambda\in[0,1]$.  Is the function $\sqrt{u(\lambda,x)}$ a concave function of the pair of variables $(\lambda, x)$?
\end{question}

A related result on the log-concavity of the first eigenfunction and eigenvalue of the Ornstein--Uhlenbeck operator $\ms{L}_{\gamma_n}$ was recently proven by Colesanti, Francini, Livshyts and Salani in \cite{CFLS24}.


\subsection{Negative exponents} Throughout this paper, we have studied the concavity  of the weighted marginal function \eqref{eq:varphi} for $\beta>0$, where $\Phi:\Omega\to\R$ is a concave function. This corresponds to the study of $\kappa$-concave measures in Borell's hierarchy \cite{Bor75}, where $\kappa\in(0,\tfrac{1}{n})$.  We pose the following question regarding weighted marginals of even $\kappa$-concave measures for $\kappa<0$.

\begin{question} \label{q:convex}
Let $\mu$ be an even log-concave measure on $\R^n$. Moreover, let $\Omega\subseteq \R^{n+1}$ be a convex set such that for each $t\in\R$, the set $\Omega_t=\{x\in\R^n: \ (t,x)\in\Omega\}$ is symmetric, and $\Phi:\Omega\to\R_+$ a convex function such that for each $t$, the function $\Phi(t,\cdot)$ is even on $\Omega_t$.  Is the function given by 
\begin{equation}
 \varphi(t) \eqdef \Bigg( \rmint_{\Omega_t} \Phi(t,x)^{-\beta}\,\diff\mu(x)\Bigg)^{-\frac{1}{\beta-n}}
\end{equation}
convex, for every $\beta>n$?
\end{question}

Similarly to the regime $\beta>0$ discussed in the introduction, when $\mu=\diff x$ is the Lebesgue measure, a positive answer to this question without any symmetry assumptions follows from~\eqref{eq:bbl2} applied to convex functions
(see also \cite{Ngu14a} for a local proof).  We note in passing that we are not aware of any examples of even $\kappa$-concave measures with $\kappa<0$ not satisfying the B-inequality or the dimensional Brunn--Minkowski inequality for symmetric convex sets. On the contrary, in \cite{CR23} it was shown that a broad family of such measures in fact satisfy both these inequalities.

We note also that our proof of Theorem \ref{thm:bbl} heavily used that the measure $\Phi(t,\cdot)^\beta \, \diff\mu$ is more log-concave than $\mu$, via the result of \cite{CR23}, so we expect some new ideas (perhaps in the form of new spectral inequalities) to be needed in order to answer Question \ref{q:convex} affirmatively.


\section*{Appendix.  Brascamp and Lieb's proof of the Gaussian B-theorem} \label{App}

As proposed in Remark~\ref{rem:B-BL} above, we end this paper by presenting the beautiful argument of Brascamp and Lieb \cite{BL75}, proving the strong version of the Gaussian B-theorem. In its functional form, this asserts that for every $n\times n$ diagonal matrix $D$ and every even convex function $V:\R^n\to\R\cup\{\infty\}$, the function given by
\begin{equation}
\forall \ t\in\R,\qquad \alpha(t) \eqdef \rmint_{\R^n} e^{-V(e^{tD}x) - |x|^2/2}\,\frac{\diff x}{(2\pi)^{n/2}}
\end{equation}
is log-concave on $\R$. Performing a change of variables, this is equivalent to the log-concavity of
\begin{equation}
\forall \ t\in\R,\qquad \beta(t)  \eqdef \rmint_{\R^n} e^{-|e^{-tD}x|^2/2-V(x)} \,\diff x,
\end{equation}
which, in turn, by scaling is equivalent to the local inequality $\beta'(0)^2\geq \beta(0)\beta''(0)$. Direct computation (see, e.g., \cite[pp.~413--414]{CFM04}) shows that this inequality can be written compactly as
\begin{equation} \label{eq:local-b}
\rmint_{\R^n} \langle x, Dx\rangle^2\,\diff\mu(x) - \Bigg( \rmint_{\R^n} \langle x,Dx\rangle\,\diff\mu(x)\Bigg)^2 \leq 2\rmint_{\R^n} \langle x,D^2x\rangle\,\diff\mu(x),
\end{equation}
where $\diff\mu(x)$ is the even probability measure with density \mbox{proportional to $e^{-V(x)}$ with respect to $\gamma_n$.}

\begin{proof} [Brascamp and Lieb's proof of \eqref{eq:local-b}]
Starting from the left-hand side, if $q(x) = \langle x,Dx\rangle$,
\begin{equation}
\mathrm{Var}_\mu( q) = \frac{1}{2Z^2} \rmint_{\R^n\times\R^n} \big( \langle x,Dx\rangle - \langle y,Dy\rangle \big)^2\ e^{-V(x)-V(y)} \,\diff\gamma_{2n}(x,y),
\end{equation}
where $Z=\rmint_{\R^n} e^{-V}\,\diff\gamma_n$ is a normalizing constant. The orthogonal transformation
\begin{equation} \label{eq:rotate}
(x,y) = \Big(\frac{u+v}{\sqrt{2}},\frac{u-v}{\sqrt{2}}\Big)
\end{equation}
preserves the $2n$-dimensional Gaussian measure, and thus we can rewrite the variance as
\begin{equation} \label{eq:apply-rot}
\begin{split}
\mathrm{Var}_\mu(q) & = \frac{2}{Z^2} \rmint_{\R^n\times\R^n}  \langle u,Dv\rangle^2 \ e^{-V(\frac{u+v}{\sqrt{2}})-V(\frac{u-v}{\sqrt{2}})} \,\diff\gamma_{2n}(u,v) 
\\ & = \frac{2}{Z^2} \rmint_{\R^n}\Bigg( \rmint_{\R^n}  \langle u,Dv\rangle^2 \, \diff \nu_v(u) \Bigg) Z(v) \,\diff\gamma_n(v),
\end{split}
\end{equation}
where $\diff\nu_v(u)$ is the probability measure on $\R^n$ with density proportional to $\exp(-V(\frac{u+v}{\sqrt{2}})-V(\frac{u-v}{\sqrt{2}}))$ with respect to $\diff\gamma_{n}(u)$ and $Z(v) = \rmint_{\R^n}\exp(-V(\frac{u+v}{\sqrt{2}})-V(\frac{u-v}{\sqrt{2}}))\,\diff\gamma_n(u)$ is a normalizing constant. Clearly every $\nu_v$ is an even measure on $\R^n$ and moreover, direct computation of the Hessian of the potential shows that $\nu_v$ is log-concave with respect to $\gamma_{n}$. Therefore, applying the Brascamp--Lieb inequality \eqref{eq:bl-var} to the odd function $\ell_v(u) = \langle u,Dv\rangle$, for every $v\in\R^n$, we have the inequality
\begin{equation}
\rmint_{\R^n}  \langle u,Dv\rangle^2 \, \diff \nu_v(u) = \mathrm{Var}_{\nu_v}(\ell_v) \leq \rmint_{\R^n} |\nabla \ell_v|^2 \,\diff\nu_v = |Dv|^2.
\end{equation}
Plugging this bound back into \eqref{eq:apply-rot} and using the definition of $Z(v)$, we finally get
\begin{equation} \label{eq:used-BL}
\begin{split}
\mathrm{Var}_\mu(q) & \leq \frac{2}{Z^2} \rmint_{\R^n\times\R^n}  |Dv|^2 \ e^{-V(\frac{u+v}{\sqrt{2}})-V(\frac{u-v}{\sqrt{2}})} \,\diff\gamma_{2n}(u,v)
\\ & =  \frac{2}{Z^2} \rmint_{\R^n\times\R^n}  \langle v,D^2v\rangle \ e^{-V(\frac{u+v}{\sqrt{2}})-V(\frac{u-v}{\sqrt{2}})} \,\diff\gamma_{2n}(u,v).
\end{split}
\end{equation}
Applying again the change of variables \eqref{eq:rotate} for which $v=\frac{x-y}{\sqrt{2}}$, we can rewrite \eqref{eq:used-BL} as
\begin{equation}
\begin{split}
\mathrm{Var}_\mu(q) & \leq \frac{2}{Z^2} \rmint_{\R^n\times\R^n} \frac{1}{2} \big( \langle x,D^2x\rangle + \langle y, D^2y\rangle - \langle x, D^2y\rangle-\langle y,D^2 x\rangle \big)\ e^{-V(x)-V(y)} \,\diff\gamma_{2n}(x,y)
\\ & = 2\rmint_{\R^n} \langle x,D^2x\rangle\,\diff\mu(x),
\end{split}
\end{equation}
where in the last equality we used that $\rmint_{\R^n} \langle x,D^2y\rangle \, \diff\mu(y) = 0$ for every $x\in\R^n$.
\end{proof}


\bibliographystyle{plain}
\bibliography{BBL-weighted}

\begin{thebibliography}{10}

\bibitem{AL25}
Gautam Aishwarya and Dongbin Li.
\newblock Entropy and functional forms of the dimensional {B}runn--{M}inkowski
  inequality in {G}auss space.
\newblock Preprint available at \url{https://arxiv.org/abs/2504.03114}, 2025.

\bibitem{AR23}
Gautam Aishwarya and Liran Rotem.
\newblock New {B}runn--{M}inkowski and functional inequalities via convexity of
  entropy.
\newblock Preprint available at \url{https://arxiv.org/abs/2311.05446}, 2023.

\bibitem{BGL14}
Dominique Bakry, Ivan Gentil, and Michel Ledoux.
\newblock {\em Analysis and geometry of {M}arkov diffusion operators}, volume
  348 of {\em Grundlehren der mathematischen Wissenschaften [Fundamental
  Principles of Mathematical Sciences]}.
\newblock Springer, Cham, 2014.

\bibitem{BBN03}
Keith Ball, Franck Barthe, and Assaf Naor.
\newblock Entropy jumps in the presence of a spectral gap.
\newblock {\em Duke Math. J.}, 119(1):41--63, 2003.

\bibitem{BC13}
Franck Barthe and Dario Cordero-Erausquin.
\newblock Invariances in variance estimates.
\newblock {\em Proc. Lond. Math. Soc. (3)}, 106(1):33--64, 2013.

\bibitem{BL00}
Sergey~G. Bobkov and Michel Ledoux.
\newblock From {B}runn-{M}inkowski to {B}rascamp-{L}ieb and to logarithmic
  {S}obolev inequalities.
\newblock {\em Geom. Funct. Anal.}, 10(5):1028--1052, 2000.

\bibitem{BL09a}
Sergey~G. Bobkov and Michel Ledoux.
\newblock Weighted {P}oincar\'e-type inequalities for {C}auchy and other convex
  measures.
\newblock {\em Ann. Probab.}, 37(2):403--427, 2009.

\bibitem{Bor75}
Christer Borell.
\newblock Convex set functions in {$d$}-space.
\newblock {\em Period. Math. Hungar.}, 6(2):111--136, 1975.

\bibitem{Bor83}
Christer Borell.
\newblock Capacitary inequalities of the {B}runn-{M}inkowski type.
\newblock {\em Math. Ann.}, 263(2):179--184, 1983.

\bibitem{Bor84}
Christer Borell.
\newblock Hitting probabilities of killed {B}rownian motion: a study on
  geometric regularity.
\newblock {\em Ann. Sci. \'Ecole Norm. Sup. (4)}, 17(3):451--467, 1984.

\bibitem{Bor85}
Christer Borell.
\newblock Greenian potentials and concavity.
\newblock {\em Math. Ann.}, 272(1):155--160, 1985.

\bibitem{Bor00}
Christer Borell.
\newblock Diffusion equations and geometric inequalities.
\newblock {\em Potential Anal.}, 12(1):49--71, 2000.

\bibitem{BK22}
K\'aroly~J. B\"or\"oczky and Pavlos Kalantzopoulos.
\newblock Log-{B}runn-{M}inkowski inequality under symmetry.
\newblock {\em Trans. Amer. Math. Soc.}, 375(8):5987--6013, 2022.

\bibitem{BLYZ12}
K\'aroly~J. B\"or\"oczky, Erwin Lutwak, Deane Yang, and Gaoyong Zhang.
\newblock The log-{B}runn-{M}inkowski inequality.
\newblock {\em Adv. Math.}, 231(3-4):1974--1997, 2012.

\bibitem{BL75}
Herm~Jan Brascamp and Elliott~H. Lieb.
\newblock Some inequalities for {G}aussian measures and the long-range order of
  the one-dimensional plasma.
\newblock In {\em Functional Integration and its Applications, Proceedings of
  the Conference on Functional Integration, Cumberland Lodge, England, edited
  by A.~M.~Arthurs}, pages 1--14. Clarendon Press, 1975.

\bibitem{BL76}
Herm~Jan Brascamp and Elliott~H. Lieb.
\newblock On extensions of the {B}runn-{M}inkowski and {P}r\'{e}kopa-{L}eindler
  theorems, including inequalities for log concave functions, and with an
  application to the diffusion equation.
\newblock {\em J. Functional Analysis}, 22(4):366--389, 1976.

\bibitem{Col05}
Andrea Colesanti.
\newblock Brunn-{M}inkowski inequalities for variational functionals and
  related problems.
\newblock {\em Adv. Math.}, 194(1):105--140, 2005.

\bibitem{Col08}
Andrea Colesanti.
\newblock From the {B}runn-{M}inkowski inequality to a class of
  {P}oincar\'{e}-type inequalities.
\newblock {\em Commun. Contemp. Math.}, 10(5):765--772, 2008.

\bibitem{CFLS24}
Andrea Colesanti, Elisa Francini, Galyna Livshyts, and Paolo Salani.
\newblock The {B}runn--{M}inkowski inequality for the first eigenvalue of the
  {O}rnstein--{U}hlenbeck operator and log-concavity of the relevant
  eigenfunction.
\newblock Preprint available at \url{https://arxiv.org/abs/2407.21354}, 2024.

\bibitem{CLM17}
Andrea Colesanti, Galyna~V. Livshyts, and Arnaud Marsiglietti.
\newblock On the stability of {B}runn-{M}inkowski type inequalities.
\newblock {\em J. Funct. Anal.}, 273(3):1120--1139, 2017.

\bibitem{Cor05}
Dario Cordero-Erausquin.
\newblock On {B}erndtsson's generalization of {P}r\'ekopa's theorem.
\newblock {\em Math. Z.}, 249(2):401--410, 2005.

\bibitem{CFM04}
Dario Cordero-Erausquin, Matthieu Fradelizi, and Bernard Maurey.
\newblock The ({B}) conjecture for the {G}aussian measure of dilates of
  symmetric convex sets and related problems.
\newblock {\em J. Funct. Anal.}, 214(2):410--427, 2004.

\bibitem{CK12}
Dario Cordero-Erausquin and Bo'az Klartag.
\newblock Interpolations, convexity and geometric inequalities.
\newblock In {\em Geometric aspects of functional analysis}, volume 2050 of
  {\em Lecture Notes in Math.}, pages 151--168. Springer, Heidelberg, 2012.

\bibitem{CR20}
Dario Cordero-Erausquin and Liran Rotem.
\newblock Several results regarding the ({B})-conjecture.
\newblock In {\em Geometric aspects of functional analysis. {V}ol. {I}}, volume
  2256 of {\em Lecture Notes in Math.}, pages 247--262. Springer, Cham, [2020]
  \copyright 2020.

\bibitem{CR23}
Dario Cordero-Erausquin and Liran Rotem.
\newblock Improved log-concavity for rotationally invariant measures of
  symmetric convex sets.
\newblock {\em Ann. Probab.}, 51(3):987--1003, 2023.

\bibitem{Din57}
Alexander Dinghas.
\newblock \"{U}ber eine {K}lasse superadditiver {M}engenfunktionale von
  {B}runn-{M}inkowski-{L}usternikschem {T}ypus.
\newblock {\em Math. Z.}, 68:111--125, 1957.

\bibitem{EM21}
Alexandros Eskenazis and Georgios Moschidis.
\newblock The dimensional {B}runn-{M}inkowski inequality in {G}auss space.
\newblock {\em J. Funct. Anal.}, 280(6):Paper No. 108914, 19, 2021.

\bibitem{Gar02}
Richard~J. Gardner.
\newblock The {B}runn-{M}inkowski inequality.
\newblock {\em Bull. Amer. Math. Soc. (N.S.)}, 39(3):355--405, 2002.

\bibitem{GZ10}
Richard~J. Gardner and Artem Zvavitch.
\newblock Gaussian {B}runn-{M}inkowski inequalities.
\newblock {\em Trans. Amer. Math. Soc.}, 362(10):5333--5353, 2010.

\bibitem{GT01}
David Gilbarg and Neil~S. Trudinger.
\newblock {\em Elliptic partial differential equations of second order}.
\newblock Classics in Mathematics. Springer-Verlag, Berlin, 2001.
\newblock Reprint of the 1998 edition.

\bibitem{HM53}
Ralph Henstock and Alexander~Murray Macbeath.
\newblock On the measure of sum-sets. {I}. {T}he theorems of {B}runn,
  {M}inkowski, and {L}usternik.
\newblock {\em Proc. London Math. Soc. (3)}, 3:182--194, 1953.

\bibitem{HL24}
Orli Herscovici and Galyna~V. Livshyts.
\newblock Kohler-{J}obin meets {E}hrhard: the sharp lower bound for the
  {G}aussian principal frequency while the {G}aussian torsional rigidity is
  fixed, via rearrangements.
\newblock {\em Proc. Amer. Math. Soc.}, 152(10):4437--4450, 2024.

\bibitem{Hil12}
David Hilbert.
\newblock {\em Grundz\"uge einer allgemeinen {T}heorie der linearen
  {I}ntegralgleichungen}.
\newblock B. G. Teubner, Leipzig, 1912.

\bibitem{Hor65}
Lars H\"ormander.
\newblock {$L\sp{2}$} estimates and existence theorems for the {$\bar \partial
  $}\ operator.
\newblock {\em Acta Math.}, 113:89--152, 1965.

\bibitem{Hor94}
Lars H\"ormander.
\newblock {\em Notions of convexity}, volume 127 of {\em Progress in
  Mathematics}.
\newblock Birkh\"auser Boston, Inc., Boston, MA, 1994.

\bibitem{HKL21}
Johannes Hosle, Alexander~V. Kolesnikov, and Galyna~V. Livshyts.
\newblock On the {$L_p$}-{B}runn-{M}inkowski and dimensional
  {B}runn-{M}inkowski conjectures for log-concave measures.
\newblock {\em J. Geom. Anal.}, 31(6):5799--5836, 2021.

\bibitem{IM24}
Mohammad~N. Ivaki and Emanuel Milman.
\newblock {$L^p$}-{M}inkowski problem under curvature pinching.
\newblock {\em Int. Math. Res. Not. IMRN}, (10):8638--8652, 2024.

\bibitem{KL21}
Alexander~V. Kolesnikov and Galyna~V. Livshyts.
\newblock On the {G}ardner-{Z}vavitch conjecture: symmetry in inequalities of
  {B}runn-{M}inkowski type.
\newblock {\em Adv. Math.}, 384:Paper No. 107689, 23, 2021.

\bibitem{KL22}
Alexander~V. Kolesnikov and Galyna~V. Livshyts.
\newblock On the local version of the {L}og-{B}runn--{M}inkowski conjecture and
  some new related geometric inequalities.
\newblock {\em Int. Math. Res. Not. IMRN}, (18):14427--14453, 2022.

\bibitem{KM17}
Alexander~V. Kolesnikov and Emanuel Milman.
\newblock Brascamp-{L}ieb-type inequalities on weighted {R}iemannian manifolds
  with boundary.
\newblock {\em J. Geom. Anal.}, 27(2):1680--1702, 2017.

\bibitem{KM18}
Alexander~V. Kolesnikov and Emanuel Milman.
\newblock Poincar\'e{} and {B}runn-{M}inkowski inequalities on the boundary of
  weighted {R}iemannian manifolds.
\newblock {\em Amer. J. Math.}, 140(5):1147--1185, 2018.

\bibitem{KM22}
Alexander~V. Kolesnikov and Emanuel Milman.
\newblock Local {$L^p$}-{B}runn-{M}inkowski inequalities for {$p<1$}.
\newblock {\em Mem. Amer. Math. Soc.}, 277(1360):v+78, 2022.

\bibitem{Lat02}
Rafa{\l} Lata{\l}a.
\newblock On some inequalities for {G}aussian measures.
\newblock In {\em Proceedings of the {I}nternational {C}ongress of
  {M}athematicians, {V}ol. {II} ({B}eijing, 2002)}, pages 813--822. Higher Ed.
  Press, Beijing, 2002.

\bibitem{Lei72b}
L\'aszl\'o Leindler.
\newblock On a certain converse of {H}\"older's inequality. {II}.
\newblock {\em Acta Sci. Math. (Szeged)}, 33(3-4):217--223, 1972.

\bibitem{Lie02}
Elliott~H. Lieb.
\newblock {\em Inequalities}.
\newblock Springer-Verlag, Berlin, 2002.
\newblock Selecta of Elliott H. Lieb, Edited, with a preface and commentaries,
  by M. Loss and M. B.\ Ruskai.

\bibitem{LMNZ17}
Galyna Livshyts, Arnaud Marsiglietti, Piotr Nayar, and Artem Zvavitch.
\newblock On the {B}runn-{M}inkowski inequality for general measures with
  applications to new isoperimetric-type inequalities.
\newblock {\em Trans. Amer. Math. Soc.}, 369(12):8725--8742, 2017.

\bibitem{Liv23}
Galyna~V. Livshyts.
\newblock A universal bound in the dimensional {B}runn-{M}inkowski inequality
  for log-concave measures.
\newblock {\em Trans. Amer. Math. Soc.}, 376(9):6663--6680, 2023.

\bibitem{Liv24}
Galyna~V. Livshyts.
\newblock On a conjectural symmetric version of {E}hrhard's inequality.
\newblock {\em Trans. Amer. Math. Soc.}, 377(7):5027--5085, 2024.

\bibitem{MD10}
Li~Ma and Sheng-Hua Du.
\newblock Extension of {R}eilly formula with applications to eigenvalue
  estimates for drifting {L}aplacians.
\newblock {\em C. R. Math. Acad. Sci. Paris}, 348(21-22):1203--1206, 2010.

\bibitem{Mar16}
Arnaud Marsiglietti.
\newblock On the improvement of concavity of convex measures.
\newblock {\em Proc. Amer. Math. Soc.}, 144(2):775--786, 2016.

\bibitem{Mau91}
Bernard Maurey.
\newblock Some deviation inequalities.
\newblock {\em Geom. Funct. Anal.}, 1(2):188--197, 1991.

\bibitem{Mil21}
Emanuel Milman.
\newblock Centro-affine differential geometry and the log-{M}inkowski problem.
\newblock To appear in {\it J. Eur. Math. Soc.}. Preprint available at
  \url{https://arxiv.org/abs/2104.12408}, 2021.

\bibitem{NT13}
Piotr Nayar and Tomasz Tkocz.
\newblock A note on a {B}runn-{M}inkowski inequality for the {G}aussian
  measure.
\newblock {\em Proc. Amer. Math. Soc.}, 141(11):4027--4030, 2013.

\bibitem{Ngu14b}
Van~Hoang Nguyen.
\newblock Dimensional variance inequalities of {B}rascamp-{L}ieb type and a
  local approach to dimensional {P}r\'ekopa's theorem.
\newblock {\em J. Funct. Anal.}, 266(2):931--955, 2014.

\bibitem{Ngu14a}
Van~Hoang Nguyen.
\newblock A local proof of the dimensional {P}r\'ekopa's theorem.
\newblock {\em J. Math. Anal. Appl.}, 419(1):20--27, 2014.

\bibitem{Pre71}
Andr\'as Pr\'ekopa.
\newblock Logarithmic concave measures with application to stochastic
  programming.
\newblock {\em Acta Sci. Math. (Szeged)}, 32:301--316, 1971.

\bibitem{Pre73}
Andr\'as Pr\'ekopa.
\newblock On logarithmic concave measures and functions.
\newblock {\em Acta Sci. Math. (Szeged)}, 34:335--343, 1973.

\bibitem{Rei77}
Robert~C. Reilly.
\newblock Applications of the {H}essian operator in a {R}iemannian manifold.
\newblock {\em Indiana Univ. Math. J.}, 26(3):459--472, 1977.

\bibitem{Rin76}
Yosef Rinott.
\newblock On convexity of measures.
\newblock {\em Ann. Probability}, 4(6):1020--1026, 1976.

\bibitem{RX21}
Michael Roysdon and Sudan Xing.
\newblock On {$L_p$}-{B}runn-{M}inkowski type and {$L_p$}-isoperimetric type
  inequalities for measures.
\newblock {\em Trans. Amer. Math. Soc.}, 374(7):5003--5036, 2021.

\bibitem{Sal15}
Paolo Salani.
\newblock Combination and mean width rearrangements of solutions to elliptic
  equations in convex sets.
\newblock {\em Ann. Inst. H. Poincar\'e{} C Anal. Non Lin\'eaire},
  32(4):763--783, 2015.

\bibitem{Sar16}
Christos Saroglou.
\newblock More on logarithmic sums of convex bodies.
\newblock {\em Mathematika}, 62(3):818--841, 2016.

\end{thebibliography}

\end{document}